\documentclass{article}
\usepackage{amssymb}
\usepackage{amsmath}

\usepackage{float}

\usepackage{amsmath} 
\usepackage[mathscr]{eucal}
\usepackage{amssymb} 
\usepackage{theorem}
\usepackage{mathtools}
\usepackage{comment}
\usepackage[shortlabels]{enumitem}

\usepackage{tikz-cd}
\usepackage{xparse}
\NewDocumentCommand{\xrightarrows}{ O{}O{} }{%
\vcenter{\hbox{%
\begin{tikzpicture}
  \node[minimum width=1cm,minimum height=1ex,anchor=south,align=center] (a){\text{\vphantom{hg}#1}\\[0.5ex] \vphantom{hg}#2};
  \draw[<-] ([yshift=0.35ex]a.west) -- ([yshift=0.35ex]a.east);
  \draw[->] ([yshift=-0.35ex]a.west) -- ([yshift=-0.35ex]a.east);
\end{tikzpicture}
}}%
}%
\usepackage{xcolor}
\usepackage{graphicx}
\usepackage{tikz}
\usepackage{eso-pic}
\theoremstyle{change}  
\newtheorem{theorem}{Theorem.}[section] 
\newtheorem{lem}[theorem]{Lemma.}  
\newtheorem{prop}[theorem]{Proposition.}
\newtheorem{coro}[theorem]{Corollary.}
\theorembodyfont{\rmfamily}  
\newtheorem{remarki}[theorem]{Remark.}
\newtheorem{example}[theorem]{Example.}

\newtheorem{defi}[theorem]{Definition.}

\newtheorem{nothing}[theorem]{} 

\newenvironment{proof}{\paragraph{Proof}}{\hfill$\square$}

\renewcommand{\le}{\leqslant} 

\newcommand{\Iso}{\mathrm{Iso}}

\newcommand{\Aut}{\mathrm{Aut}}
\newcommand{\calA}{\mathcal{A}}

\newcommand{\calC}{\mathcal{C}}
                       
\newcommand{\calE}{\mathcal{E}}
\newcommand{\calF}{\mathcal{F}}

\newcommand{\calM}{\mathcal{M}}

\newcommand{\CC}{\mathbb{C}}

\newcommand{\Def}{\mathrm{Def}}
\newcommand{\Defres}{\mathrm{Defres}}

\newcommand{\Ext}{\mathrm{Ext}}

\newcommand{\Hom}{\mathrm{Hom}}

\newcommand{\id}{\mathrm{id}}

\newcommand{\Ind}{\mathrm{Ind}}

\newcommand{\Inf}{\mathrm{Inf}}
\newcommand{\Indinf}{\mathrm{Indinf}}
\newcommand{\Inn}{\mathrm{Inn}}

\newcommand{\Out}{\mathrm{Out}}

\newcommand{\Res}{\mathrm{Res}}

\newcommand{\ZZ}{\mathbb{Z}}

\newcommand{\tildeeKk}{\widetilde{e_{(K, \kappa)}^G}}
\newcommand{\tildeeLl}{\widetilde{e_{(L, \lambda)}^G}}
\newcommand{\tildeeG}{\widetilde{e_G^G}}
\newcommand{\ctimes}{{\mathbb C^\times}}
\newcommand{\1}{{\bf 1}}

\title{A Categorical Decomposition of $\CC^{\times}$-fibered $p$-biset Functors} 

\author{Olcay Co\c{s}kun$^1$ and Ruslan Muslumov$^2$}

\date{
	$^1$ Mathematics Research Center, ASOIU, Baku, Azerbaijan \\ 
   \texttt{olcay.coshkun@asoiu.edu.az}\\%
	$^2$ ADA University, Baku, Azerbaijan\\ \texttt{rmuslumov@ada.edu.az} \\[2ex]%
}

\begin{document}
\maketitle
\begin{abstract}
We generalize Bouc’s construction of orthogonal idempotents in the double Burnside algebra to the setting of the double $\mathbb{C}^\times$-fibered Burnside algebra. This yields a structural decomposition of the evaluations of $\mathbb{C}^\times$-fibered biset functors on finite groups. We then construct a complete set of orthogonal idempotents in the category of $\mathbb{C}^\times$-fibered $p$-biset functors, leading to a categorical decomposition of this category into subcategories indexed by isomorphism classes of atoric $p$-groups.

Furthermore, we introduce the notion of \emph{vertices} for indecomposable functors and establish that the $\mathrm{Ext}$-groups between simple functors with distinct vertices vanish. As an application, we describe a set containing composition factors of the monomial Burnside functor, thereby providing new insights into its structure. Additionally, we develop a technique for analyzing fibered biset functors via their underlying biset structures.


\smallskip


\end{abstract}


\section{Introduction}

After its introduction in \cite{densembles}, Bouc's  theory of biset functors has become a central framework for studying operations coming from permutation actions on representation theoretic constructions. In particular, the classification and decomposition of $p$-biset functors over a field $k$ of characteristic different from $p$ has been extensively developed by Serge Bouc and others, see \cite{biset functor} and references there. A cornerstone of this theory is the construction of canonical orthogonal idempotents in the double Burnside algebra $kB(G,G)$, indexed by minimal sections of $G$, which allows for a fine decomposition of functor evaluations and a block-theoretic understanding of the category of $p$-biset functors.

A generalization of biset functors, introduced by Robert Boltje and the first author \cite{fibered biset} incorporates multiplicative character data and yields a more refined structure. Although the new theory of fibered biset functors is more involved, it has found several applications and has become a useful tool in representation theory, see \cite{CY}, \cite{BY}, \cite{BGE}, \cite{C}.  

In this paper, we focus on $\mathbb C^\times$-fibered biset functors only and extend Bouc's results from \cite{double} to the fibered setting. Our goal is to construct and analyze a family of orthogonal idempotents in the fibered double Burnside algebra $B_k^{\mathbb C^\times}(G,G)$ that mirrors the structure and properties of Bouc's idempotents in the classical setting.

To achieve this, we develop a truncated algebra $\mathcal E(G)$ modeled on Bouc’s construction but adapted to the case of fibered bisets. We define a canonical basis of $\mathcal E(G)$ using bisets of the form $[G \times G / (U, \upsilon)]$ where $U$ is a covering subgroup, that is, it has full coordinate projections, and $\upsilon$ is a linear character, and we provide explicit formulas for the multiplication of these basis elements. We then construct a system of idempotents $\varphi^G_{K,\kappa}$, indexed by pairs $(K,\kappa)$, where $K \leq \Phi(G)$ is normal in $G$ and $\kappa: K\to\ctimes$ is a character,  and show that these idempotents are mutually orthogonal and sum to the identity element of $\mathcal E(G)$.

The idempotents $\varphi_{K, \kappa}^G$ replaces the idempotents $\varphi_N^G$ for $N\unlhd G$ in \cite{double} with a twist. In Bouc's framework, every idempotent $\varphi_N^G$ can be expressed in terms of $\varphi_1^G$, highlighting the distinguished role played by the latter within the theory. Due to the fibered structure, there is not only one special idempotent but a set of them, precisely, we can express all $\varphi_{K, \kappa}^G$ in terms of the ones where the character $\kappa$ is faithful. Nevertheless we still obtain a decomposition of fibered biset functors via this set of idempotents. Namely, for any fibered biset functor
$F$ and any finite group $G$ we obtain a decomposition 
\[
F(G) \cong \bigoplus_{(T,S) \in [\Pi(G)]} \bigoplus_{(K,\kappa) \in [\mathcal{M}_{T/S}^*]} \left( \varphi_{[K,\kappa]}^{T/S} F(T/S) \right)^{N_G(T,S)/T}.
\]
Here, $[\Pi(G)]$ denotes a set of representatives of the conjugacy classes of minimal sections of $G$, i.e., sections \(S\unlhd T\leq G\) with \(S \leq \Phi(T)\), and $[\mathcal{M}_{T/S}^*]$ denotes a set of orbit representatives $[K,\kappa]$ under the action of $\Out(T/S)$, where $K$ is a normal subgroup of $T/S$ contained in its Frattini subgroup, and $\kappa$ is a faithful character of $K$.

The main result of the paper is a decomposition theorem for  the category $\mathcal F_{k,p}^{\mathbb C^\times}$ of $\mathbb{C}^\times$-fibered $p$-biset functors. Assuming that the prime $p$ is invertible in the coefficient ring $k$, we show that there is an equivalence of categories
\[
\mathcal F_{k,p}^{\mathbb C^\times} \simeq \prod_{M\in[\mathcal At_p]}\widehat{c}_M\mathcal F_{k,p}^{\mathbb C^\times},
\]
where the product is over isomorphism classes of atoric $p$-groups and for a given atoric $p$-group $M$, we denote by $\widehat{c}_M$ the \emph{Bouc's idempotent} as defined in Section \ref{Decomposing}. These are certain central idempotents of the category $\mathcal F_{k,p}^{\mathbb C^\times}$. This is the fibered version of Bouc's decomposition theorem from \cite{double}.

As an application of the above theorem, we define vertices of indecomposable $\mathbb C^\times$-fibered biset functors and show that ext-groups of simple functors with distinct vertices vanishes. This in turn leads to a new approach to the problem of determining biset functor structure of a fibered biset functor. 
In this paper, we have shown that for an indecomposable fibered $p$-biset functor $F$ with vertex $M$, composition factors of the underlying $p$-biset functor $F^\flat$ are controlled by $M$ in a particular way. See Section \ref{Applications} for details.

\section{Preliminaries}\label{sec:prelim}

In this section we collect the basic definitions and results about fibered bisets and fibered biset functors that will be used throughout the paper. We refer the reader to \cite{fibered biset} for further details. Fix a multiplicatively written abelian group $A$ and a commutative ring $k$.

\begin{nothing}{\bf Fibered bisets.} Let $G$ and $H$ be finite groups. An \emph{$A$-fibered $G$-set} is a left $(A \times G)$-set that is free as an $A$-set and has finitely many $A$-orbits. More generally, an \emph{$A$-fibered $(G,H)$-biset} is an $A$-free $(A \times G \times H)$-set with finitely many $A$-orbits. A morphism between $A$-fibered $(G,H)$-bisets is an $(A\times G \times H)$-equivariant map.

Given an $A$-fibered $(G,H)$-biset $X$ and $x \in X$, the \emph{stabilizing pair} at $x$ is $(U_x, \phi_x)$, where $U_x \leq G \times H$ is the stabilizer of the $A$-orbit $A\cdot x$ under the induced $G\times H$-action, and $\phi_x: U_x \to A$ is the unique map satisfying $u\cdot x = \phi_x(u) x$.

An $A$-fibered $(G,H)$-biset $X$ is called \emph{transitive} if the $G\times H$-action on the set of $A$-orbits is transitive. Each transitive $A$-fibered $(G,H)$-biset is determined up to isomorphism by any of its stabilizing pairs $(U,\phi)$. We write
\[
\mathcal{M}_{G \times H}^A = \{ (U, \phi) \mid U \leq G \times H,\ \phi \in \Hom(U, A) \}.
\]
When the fiber group \(A\) is clear from context, we abbreviate \(\Hom(U, A)\) by \(U^*\), and we write \(\mathcal{M}_{G\times H}\) instead of \(\mathcal{M}_{G\times H}^A\). For a given pair $(U,\upsilon)\in\mathcal M_{G\times H}$, we denote by
\[
\Big(\frac{G\times H}{U, \upsilon}\Big) \quad \text{and}\quad \Big[ \frac{G\times H}{U,\upsilon} \Big]
\]
the corresponding transitive $A$-fibered $(G,H)$-biset and its isomorphism class.

By Goursat Theorem, we attach the quintuple $(p_1(U), k_1(U), \phi, k_2(U), p_2(U)$ to the subgroup $U$. Here $p_1(U), p_2(U)$ are the first and the second projections of $U$, respectively. We also have $k_1(U)=\{k\in G\mid (k,1)\in U\}$ and $k_2(U)=\{l\in H\mid (1,l)\in U\}$ and $\phi: p_2(U)/k_2(U)\to p_1(U)/k_1(U)$ is the isomorphism determined by $U$. We also write the restriction of \(\upsilon\) to $k_1(U)\times k_2(U)$ as \(\upsilon|_{k_1(U)\times k_2(U)} = \upsilon_1 \times \upsilon_2^{-1}\) for some $\upsilon_1\in k_1(U)^*$ and $\upsilon_2\in k_2(U)^*$. Following \cite[Section 1]{fibered biset}, we call $l(U, \upsilon) = (p_1(U), k_1(U), \upsilon_1)$ the left invariants  and $r(U,\upsilon) = (p_2(U), k_2(U), \upsilon_2)$ the right invariants of $(U, \upsilon)$. Moreover we construct the subgroups
$\widehat{k_1(U)} = \ker\kappa$ and $\widehat{k_2(U)} = \ker\lambda$.Together with the left and right invariants, by \cite[Proposition 1.3]{fibered biset}, we have that $k_1(U)/\widehat{k_1(U)}$ is central in $p_1(U)/\widehat{k_1(U)}$ and that $k_2(U)/\widehat{k_2(U)}$ is central in $p_2(U)/\widehat{k_2(U)}$.

\begin{defi}
Let ${}_G\mathrm{set}_H^A$ denote the category of $A$-fibered $(G,H)$-bisets with morphisms given by $(A\times G\times H)$-equivariant maps. Disjoint union defines a categorical coproduct in ${}_G\mathrm{set}_H^A$.

The Grothendieck group of ${}_G\mathrm{set}_H^A$ is denoted by $B^A(G,H)$ and is called the \emph{fibered Burnside group}. A $\ZZ$-basis of $B^A(G,H)$ is given by the isomorphism classes of transitive $A$-fibered $(G,H)$-bisets.
\end{defi}

\begin{nothing} 
When $A = \{1\}$ is the trivial group, we write $B(H, G):= B^{\{1\}}(H, G)$ for the Burnside group of $(H, G)$-bisets. We refer to \cite[Part I]{biset functor}.
Throughout the paper we identify \( B_k(H, G):= k\otimes B(H, G) \) as a subgroup of $B_k^{\CC^{\times}}(H, G):= k\otimes B^{\CC^{\times}}(H, G)$ under the canonical inclusion
\[
    B_k(H, G) \hookrightarrow B_k^{\CC^{\times}}(H, G)
\]
given by 
\[
\left[  \frac{G \times H}{U} \right]\mapsto \left[ \frac{G \times H}{U, 1} \right].
\]
With this identification, we have the usual operations of \emph{induction}, \emph{restriction}, \emph{inflation}, \emph{deflation}, and \emph{isomorphism} in the fibered Burnside group. We fix our notation below. Let $H\le G$ and $N\unlhd G$. Let $\pi: G\to G/N$ denote the natural projection. Then 
\[
\Ind_H^G = \left[ \frac{G \times H}{\Delta(H), 1} \right], \quad \Res^G_H = \left[ \frac{H \times G}{\Delta(H), 1} \right].
\]
\[
\Inf_{G/N}^G = \left[ \frac{G \times G/N}{\Delta_\pi(G), 1} \right], \quad \Def^G_{G/N} = \left[ \frac{G/N \times G}{{}_\pi\Delta(G), 1} \right]
\]
where 
\[
\Delta_\pi(G) = \{(g, gN)\mid g\in G\}\le G\times G/N
\]
and 
\[
{}_\pi\Delta(G) = \{(gN, g)\mid g\in G\}\le G/N\times G.
\]
\end{nothing}

In order to define fibered biset functors, we introduce a composition product on fibered bisets. Let $G$, $H$, and $K$ be finite groups, and let $X$ be an $A$-fibered $(G,H)$-biset and $Y$ an $A$-fibered $(H,K)$-biset. The \emph{composition product}, or \emph{tensor product}, of $X$ and $Y$ is defined to be the $A$-fibered $(G,K)$-biset given by the set of $A$-free orbits in the orbit space
\[
X \times_{AH} Y := (X \times Y)/\!(A \times H),
\]
where the action of $(a,h)\in A \times H$ on $(x,y)\in X \times Y$ is given by
\[
{}^{(a,h)}(x,y) = (x(a^{-1},h^{-1}), (a,h)y).
\]

If both $X$ and $Y$ are transitive, the resulting biset admits an explicit decomposition via the Mackey formula:
\[
\left( \frac{G \times H}{U, \upsilon} \right) \circ \left( \frac{H \times K}{V, \omega} \right) 
\cong 
\bigoplus_{\substack{t \in p_2(U) \backslash H / p_1(V) \\ \upsilon_2|_{H_t} = \omega_1|_{H_t}}}
\left( \frac{G \times K}{U * {}^{(t,1)}V,\, \upsilon * {}^{(t,1)}\omega} \right),
\]
where $H_t := k_2(U) \cap {}^t k_1(V)$. Also we set
        \[
            U*V := \{(g,k)\in G\times K \mid \exists h\in H, (g,h)\in U \, \text{and} \, (h,k)\in V\}
        \]
        and
        \[
            \upsilon * \omega \in (U*V)^* \quad \text{by} \quad (\upsilon * \omega)(g,k) := \upsilon(g,h)\omega(h,k),
        \]
        where $h\in H$ is chosen such that $(g,h)\in U$ and $(h,k)\in V$. See \cite[Corollary 2.5]{fibered biset} for a proof.

Under this product structure, by \cite[Proposition 2.8]{fibered biset} any transitive $A$-fibered $(G,H)$-biset with stabilizing pair $(U,\upsilon)$ admits a canonical decomposition:
\begin{equation}
\label{basiselementdecomposition}
    \left( \frac{G \times H}{U, \upsilon} \right) 
    \cong 
    \Ind_P^G \circ \Inf_{P / \widehat{K}}^P \circ X \circ \Def_{Q / \widehat{L}}^H \circ \Res_Q^H,
\end{equation}
where $(P,K,\kappa) = l(U,\upsilon)$, $(Q,L,\lambda) = r(U,\upsilon)$, and $\widehat{K} = \ker \kappa$, $\widehat{L} = \ker \lambda$. Here, $X$ is a transitive $A$-fibered $(P/\widehat{K}, Q/\widehat{L})$-biset whose stabilizing pair $(U', \upsilon')$ satisfies $U' = U / (\widehat{K} \times \widehat{L})$ and $\upsilon' = \upsilon \circ \pi^{-1}$ where
\[
\pi: (P\times Q) / (\widehat{K} \times \widehat{L}) \to P/\widehat{K} \times Q/\widehat{L}
\]
is the canonical projection. Moreover, the restriction of $\upsilon'$ to the largest rectangular subgroup $K' \times L' \leq U'$ is product of faithful characters.
\end{nothing}

\begin{nothing}
\label{section:eGG}
    By \cite[Lemma 2.5.8]{biset functor}, there is also a ring homomorphism $\widetilde{?}$ from the Burnside ring \( B_k(G) \) of \( G \) to the ring $B_k(G, G)$ and hence to $B_k^\ctimes(G,G)$. Here the ring structure on $B_k^\ctimes(G,G)$ comes from the tensor product of fibered bisets. For any \( G \)-set \( [G/K] \in B_k(G) \), we have
    \[
        \widetilde{[G/K]} := \left[\frac{G\times G}{\Delta(K),1}\right] \in B_k^{\CC^{\times}}(G,G).
    \]
    When the order of \( G \) is invertible in the coefficient ring \( k \), for any subgroup $H\le G$, the element 
    \[
        e_H^G := \frac{1}{|N_G(H)|} \sum_{K \leq H} |K| \mu(K, H) [G/K] \in B_k(G)
    \]
    is a primitive idempotent of $B_k(G)$. Moreover 
    as $H$ runs over all subgroups of $G$ up to conjugation, they form a complete set of orthogonal primitive idempotents of $B_k(G)$. 
    The image of this element under the aforementioned ring homomorphism is the idempotent 
    \[
        \widetilde{e_H^G} := \frac{1}{|N_G(H)|} \sum_{K \leq H} |K| \mu(K, H) \left[ \frac{G \times G}{\Delta(K), 1} \right] \in B_k^{\CC^{\times}}(G,G).
    \]
    Consequently, the elements \( \widetilde{e_H^G} \), as \( H \) runs over a set of representatives of conjugacy classes of subgroups of $G$, form a system of orthogonal idempotents in the \( k \)-algebra \( B_k^{\CC^{\times}}(G,G) \), whose sum equals the identity element. 
\end{nothing}
\begin{nothing}
\label{elemopontilde}
    Let $G$ be a group and let $k$ be a commutative unitary ring in which $|G|$ is invertible. By \cite[Theorem 5.2.4 and Corollary 5.2.4]{biset functor}, one determines the interplay between induction, restriction, deflation, and inflation maps with the idempotent element $\widetilde{e_G^G}$. We recall the resulting equations below. 

    \smallskip
    
       \noindent {(a)} For any $H\lneq G$, we have 
        \begin{eqnarray*}
            \Res_H^G \widetilde{e_G^G} &=& 0, \quad \text{and} \quad \widetilde{e_G^G} \Ind_H^G = 0.
        \end{eqnarray*}
        
        \noindent (b) For $N\unlhd G$ and $N\leq \Phi(G)$, we have
        \begin{eqnarray*}
            \widetilde{e_{G/N}^{G/N}} \Def_{G/N}^G &=& \Def_{G/N}^G \widetilde{e_G^G}, \\
            \Inf_{G/N}^G \widetilde{e_{G/N}^{G/N}} &=& \widetilde{e_G^G} \Inf_{G/N}^G.
        \end{eqnarray*}
        
        \noindent (c) For any $N\unlhd G$, we have
        \begin{eqnarray*}
            \Def_{G/N}^G \widetilde{e_G^G} \Inf_{G/N}^G &=& m_{G,N}\widetilde{e_{G/N}^{G/N}}, 
        \end{eqnarray*}
        where the deflation number $m_{G, N}$ is given by
        \[
            m_{G,N} = \frac{1}{|G|} \sum_{\substack{X \le G \\ XN = G}} |X|\mu(X,G).
        \]
\end{nothing}

\begin{nothing}{\bf Fibered biset functors.}
Let $A$ be an abelian group and $k$ a commutative ring. The \emph{$A$-fibered biset category} $\mathcal{C}_k^A$ is defined as follows:
\begin{itemize}
    \item Objects are finite groups.
    \item For finite groups $G$ and $H$, the morphism set from $G$ to $H$ is given by
    \[
    \mathrm{Hom}_{\calC_k^A}(G,H) := k \otimes_\mathbb{Z} B^A(H, G) = B_k^A(H,G)  .
    \]
    \item Composition of morphisms is induced by the tensor product of $A$-fibered bisets, linearized over $k$.
\end{itemize}
We will also consider a full subcategory of \(\mathcal{C}_k^A\), whose objects are restricted to finite \(p\)-groups, where \(p\) is a prime number with \(p\in k^\times\). In this case, we denote the subcategory by \(k\mathcal{C}_p^A\).

An \emph{$A$-fibered biset functor over $k$} is a $k$-linear functor $F \colon \mathcal{C}_k^A \to k\text{-Mod}$. The category $\mathcal{F}_k^A$ of such functors is abelian. In Section~\ref{Decomposing}, we restrict our attention to \(\CC^\times\)-fibered biset functors defined on \(p\)-groups. The category of \(k\)-linear functors with domain \(k\mathcal{C}_p^{\CC^\times}\), for a prime number \(p\) with \(p \in k^\times\), is then denoted by \(\mathcal{F}_{p,k}^{\CC^\times}\).

As shown in \cite{fibered biset}, simple objects in $\mathcal{F}_k^A$ are determined by their evaluations at minimal groups. The full classification given in \cite[Theorem 9.2]{fibered biset}, however, is intricate and requires the introduction of substantial additional notation. We briefly recall the relevant structure below, omitting proofs and technical details.

Let $E_G := B^A_k(G,G)$ denote the endomorphism algebra of $G$ in the fibered biset category, and let $I_G$ be the ideal generated by morphisms factoring through groups of strictly smaller order. The \emph{essential algebra} of $G$ is then defined by
\[
\overline{E}_G := E_G / I_G.
\]
A functor $F \in \mathcal{F}_k^A$ is simple with minimal group $G$ if and only if $F(G)$ is a simple $\overline{E}_G$-module. The classification of simple objects in $\mathcal{F}_k^A$ thus reduces to the classification of simple $\overline{E}_G$-modules. To facilitate this, we now describe a suitable classification of the simple modules over the essential algebra.

Let $\mathcal{M}_G^G := \mathcal{M}_G^{A,G}$ denote the subset of $ \mathcal{M}_G^A $ consisting of all pairs $(K,\kappa)$ such that $K \unlhd G$ and $\kappa \in \mathrm{Hom}(K,A)$ is $G$-invariant. We endow $\calM_G^G$ with a poset structure by declaring that $(K, \kappa) \leq (L, \lambda)$ in $\calM_G^G$ if and only if $K \leq L$ and $\lambda|_K = \kappa$. To each such pair, one associates an element of the fibered Burnside ring
\[
e_{(K,\kappa)}^G := \left[ \frac{G \times G}{\Delta_K(G), \phi_\kappa} \right],
\]
where $\Delta_K(G) = \{(g,h) \in G \times G : gK = hK\}$ and $\phi_\kappa(g,h) := \kappa(h^{-1}g)$. These elements lie in $E_G$ and form a family of idempotents with favorable multiplicative properties.

We record the key multiplication rule for these elements, which plays an important role in the decomposition of the essential algebra:
\begin{prop}\cite[Section 4.3]{fibered biset}
\label{eKtimeseL}
Let $(K,\kappa), (L,\lambda) \in \mathcal{M}_G^G$. Then
\[
e_{(K,\kappa)}^G \cdot e_{(L,\lambda)}^G =
\begin{cases}
e_{(KL,\, \kappa \lambda)}^G, & \text{if } \kappa|_{K \cap L} = \lambda|_{K \cap L}, \\
0, & \text{otherwise}.
\end{cases}
\]
\end{prop}

The multiplication rule above gives rise to a family of orthogonal idempotents as follows. For each $(K, \kappa) \in \mathcal{M}_G^{G}$, define
\[
f_{(K, \kappa)}^G := e_{(K, \kappa)}^G - \sum_{(K', \kappa') > (K, \kappa)} \mu^{\triangleleft}_{K, K'}\, e_{(K', \kappa')}^G,
\]
where the partial order is given by $(K, \kappa) < (K', \kappa')$ if $K < K'$ and $\kappa'|_K = \kappa$ and where $\mu^\triangleleft_{K, K'}$ denotes the Möbius function of the poset of normal subgroups of $G$. These idempotents will be used in Section \ref{Decomposing}.

A pair $(K,\kappa) \in \mathcal{M}_G^{G}$ is called \emph{reduced} if $e_{(K, \kappa)}$ does not lie in the ideal $I_G \subseteq B^A(G,G)$ generated by morphisms factoring through proper subquotients of $G$. Equivalently, $(K,\kappa)$ is reduced if $f_{(K,\kappa)}^G \ne 0$ in the essential algebra $\overline{E}_G$.

When $A = \mathbb{C}^\times$, it is shown in \cite[Corollary 10.13]{fibered biset} that a pair $(K, \kappa)$ is reduced if and only if $K \leq Z(G) \cap [G,G]$ is cyclic and $\kappa$ is faithful. We denote the set of reduced pairs in $\mathcal{M}_G^G$ by $\mathcal{R}_G$.

To each pair $(K,\kappa) \in \mathcal{M}_G^G$, one associates a finite group
$\Gamma_{(G,K,\kappa)}$,
defined as the group of transitive $A$-fibered $(G,G)$-bisets whose left and right invariants are both equal to $(G, K, \kappa)$. 

With this notation and by \cite[Corollary~8.5]{fibered biset}, there is a bijection between the set of isomorphism classes of simple $\overline{E}_G$-modules and the set
\[
\left\{ (K, \kappa, [V]) \ \middle| \ (K, \kappa) \in \widetilde{\mathcal{R}}_G,\ [V] \in \mathrm{Irr}(k\Gamma_{(G,K,\kappa)}) \right\},
\]
where $\widetilde{\mathcal{R}}_G$ denotes the set of \emph{linkage classes} of reduced pairs in $\mathcal{R}_G$. The precise definition of the linkage relation is not needed in this paper; we refer to \cite[Section~8]{fibered biset} for details.

Now consider the set
\[
\mathcal{S} := \left\{ (G, K, \kappa, [V]) \ \middle| \ G \in \mathrm{Ob}(\mathcal{C}_k^A),\ (K, \kappa) \in \widetilde{\mathcal{R}}_G,\ [V] \in \mathrm{Irr}(k\Gamma_{(G,K,\kappa)}) \right\}.
\]
The linkage relation extends naturally to such quadruples, and we denote the set of resulting equivalence classes by $\overline{\mathcal{S}}$. Then, by \cite[Theorem 9.2]{fibered biset} there is a bijection between the set of isomorphism classes of simple $A$-fibered biset functors and the set $\overline{\mathcal{S}}$. The simple fibered biset functor corresponding to a quadruple $(G, K, \kappa, [V])$ is denoted by $S_{(G, K, \kappa, [V])}.$

\end{nothing}



\section{Truncated Algebra $\calE(G)$}
Following Bouc’s approach to decomposing the double Burnside algebra via certain idempotents, we consider a truncated algebra $\mathcal{E}(G)$ in the fibered setting. This allows us to isolate contributions from minimal sections and construct idempotents reflecting the fibered support. In this section, we define $\mathcal{E}(G)$ and specify a $k$-basis for it. We then introduce a family of orthogonal idempotent elements whose sum is the identity element of the algebra, and compute the product of each idempotent with the basis elements.

\begin{nothing}
Let $G$ be a group and let $k$ be a commutative unitary ring such that the order of $G$ is invertible in $k$. Following Bouc, we define the truncated $k$-algebra $\mathcal{E}(G)$ by
\[
    \mathcal{E}(G) := \widetilde{e_G^G} \cdot B_k^{\mathbb{C}^\times}(G, G) \cdot \widetilde{e_G^G}.
\]
Let $(U, \upsilon)\in \calM_{G\times G}$ be such that $p_1(U) = p_2(U) = G$. For future reference we call a subgroup of $G\times H$ with this projection property a \emph{covering} subgroup. We define the element
\[
    Y_{U,\upsilon} := \widetilde{e_G^G} \left[ \frac{G \times G}{U, \upsilon} \right] \widetilde{e_G^G} \in \mathcal{E}(G).
\]
A similar algebra, $E_G^c$ was constructed in \cite[Section 6.1]{fibered biset}. It is easy to see that $\mathcal E(G) = \tildeeG\cdot  E_G^c\cdot\tildeeG$.
\end{nothing}

\begin{nothing}
\label{basis}
One can verify that the elements $Y_{U,\upsilon}$ constitute a $k$-basis of $\calE(G)$, where $(U, \upsilon)$ ranges over a set of representatives of the $(G \times G)$-conjugacy classes of covering subgroups of $G\times G$. Furthermore, explicit formulas for the products of basis elements can be found as follows. Let $(U,\upsilon)$ and $(V,\omega)$ be in $\calM_{G\times G}$ with covering subgroups $U$ and $V$ of $G \times G$. Then we have:
\[
Y_{U,\upsilon} \cdot Y_{V,\omega} = \widetilde{e_G^G}  \left( \Big [ \frac{G\times G}{U,\upsilon} \Big ]   \widetilde{e_G^G}   \Big[ \frac{G\times G}{V,\omega} \Big ] \right)  \widetilde{e_G^G}.
\]
The middle term can be evaluated:
\begin{eqnarray*}
&& \Big [ \frac{G\times G}{U,\upsilon} \Big ]  \widetilde{e_G^G}  \Big [ \frac{G\times G}{V,\omega} \Big ]=\\
&=& \frac{1}{|G|} \sum_{X\leq G} |X| \mu(X,G) \left ( \Big [\frac{G\times G}{U,\upsilon}\Big]\Big [ \frac{G\times G}{\Delta(X),1}\Big] \Big [ \frac{G\times G}{V,\omega} \Big ] \right )\\
&=& \frac{1}{|G|} \sum_{\substack{X\leq G \\ \upsilon_2 \mid _{H_0} = \omega_1 \mid _{H_0} }} |X| \mu(X,G) \left ( \Big [\frac{G\times G}{U* \Delta(X) * V,\upsilon * 1 * \omega }\Big]  \right ),
\end{eqnarray*}
where $H_0 = k_2(U)\cap X \cap k_1(V)$. It follows, by left and right multiplication by $\widetilde{e_G^G}$, that:
\[
Y_{U,\upsilon} \cdot Y_{V,\omega} = \frac{1}{|G|} \sum_{X \in \delta(U,V)} |X| \mu(X,G) Y_{U * \Delta(X) * V, \upsilon * 1 * \omega}.
\]
Here $\delta(U, V)$ denotes the set of subgroups $X$ of $G \times G$ such that $U * \Delta(X) * V$ is a covering subgroup of $G \times G$ and $\upsilon_2\mid_{H_0} = \omega_1\mid_{H_0}$. 

Bouc has shown in \cite[Section 9.2]{densembles} that there exists a bijection between the covering subgroups of the form $U * \Delta(X) * V$ in $G \times G$ and the subgroups $X$ of $G$ satisfying $X k_2(U) = X k_1(V) = G$, where $U$ and $V$ are covering subgroups of $G$. Furthermore, he established that the set of subgroups $X \leq G$ such that $X k_2(U) = X k_1(V) = G$ and $k_2(U) \cap k_1(V) \leq X$ is in bijection with the set of covering subgroups of $G \times G$ that contain $k_1(U) \cap k_2(V)$ and are contained in $U * V$.

Applying these group-theoretic results, we obtain the following theorem. We omit the proof, as it closely follows the argument in \cite[Section~9]{densembles} for bisets, which is also restated in \cite[Proposition~3.2]{double}. The only difference lies in an additional condition arising from the Mackey formula for fibered bisets.

\begin{prop} Let $(U, \upsilon), (V, \omega)\in\calM_{G\times G}$ be covering pairs. Then
\begin{equation}
Y_{U,\upsilon} \cdot Y_{V,\omega} = \frac{m_{G,H_1}}{|G|} \sum_{\substack{H_1\leq X\leq G \\ Xk_2(U) = Xk_1(V) = G\\ \upsilon_2\mid _{H_1} = \omega_1\mid _{H_1}  }} |X| \mu(X,G)  Y_{L * \Delta(X) * M, \upsilon * 1 * \omega},
\end{equation}
where $H_1 = k_2(U) \cap k_1(V)$.
    
\end{prop}

\begin{nothing}
\label{productofYs}
    If one of the groups $k_2(U)$ or $k_1(V)$ is contained in $\Phi(G)$ and $\upsilon_2\mid_{H_1} = \omega_1\mid_{H_1}$ then 
    \[
        Y_{U,\upsilon} Y_{V,\omega} = Y_{U*V,\upsilon * \omega}.
    \]  
    This identity allows us to remove Möbius-weighted summations. This is central in proving orthogonality of certain idempotents that we shall define. 
\end{nothing}
Next we examine a special case of these basis elements to form certain idempotents. Given a normal subgroup $K\unlhd G$ and $\kappa \in K^*$, we set
\[
\tildeeKk := Y_{\Delta_K(G),\phi_\kappa} = \tildeeG\cdot e_{(K, \kappa)}^G\cdot \tildeeG.
\] The following is an elementary property.
\end{nothing}

\begin{lem}
\label{etilde}
    Let \( G \) be a group. For a normal subgroup \( K \) of \( G \) contained in the Frattini subgroup of \( G \), and \( \kappa \in K^* \), we have
    \begin{equation}
        \label{tildetimesipot}
        \tildeeKk = \widetilde{e_G^G}\cdot e_{(K,\kappa)}^G = e_{(K,\kappa)}^G \cdot\widetilde{e_G^G}.
    \end{equation}
\end{lem}

\begin{proof}
    By expanding $\widetilde{e_G^G}$ the product can be written as 
    \begin{eqnarray*}
        \tildeeKk = \widetilde{e_G^G} \frac{1}{|G|}\sum_{\substack{X\leq G}}|X|\mu(X,G) \Big [ \frac{G\times G}{\Delta_K(G),\phi_{\kappa}}\Big] \Big [ \frac{G\times G}{\Delta(X),1}\Big]
    \\
        = \widetilde{e_G^G} \frac{1}{|G|}\sum_{X\leq G}|X|\mu(X,G) \Big [ \frac{G\times G}{\Delta_K(G) * \Delta(X),\phi_{\kappa}\ast 1}\Big].
    \end{eqnarray*}    
    Now $p_1(\Delta_K(G)* \Delta(X)) = XK$ and it is equal to $G$ if and only if $X=G$ since $K\leq \Phi(G)$. Moreover, $\Res_H^G \widetilde{e_G^G} = 0$ for any proper subgroup $H$ of $G$. Thus the sum collapses to the case $X=G$ and hence we get
    \begin{eqnarray*}
        \tildeeKk = \widetilde{e_G^G}\cdot e_{(K,\kappa)}^G.
    \end{eqnarray*}
    The equality $\tildeeKk = e_{(K,\kappa)}^G \cdot \widetilde{e_G^G}$ follows by taking opposite bisets as described in \cite[Section 1.6]{fibered biset}, since both $\widetilde{e_G^G}$ and $e_{(K,\kappa)}^G$ are self-opposite.
\end{proof}

\begin{lem}
\label{infdefonipot}\label{Yande}
Let $M\unlhd G$ and $(L,\lambda)\in \calM_G^G$ satisfy $L\cap M \leq \ker \lambda$. Let $\bar\lambda\in (LM/M)^*$ denote the canonical image of $\lambda$ and $\lambda'=\Inf_{LM/M}^{LM}\bar\lambda$. Then

\smallskip

\noindent (a)   $e_{(L, \lambda)}^G \Inf_{G/M}^G = \Inf_{G/M}^G e_{(LM/M,\bar{\lambda})}^{G/M}$

\smallskip

\noindent (b)   $\Def_{G/M}^G e_{(L, \lambda)}^G = e_{(LM/M,\bar{\lambda})}^{G/M}  \Def_{G/M}^G$

\smallskip

\noindent (c)
$        \Inf_{G/M}^G e_{(LM/M,\bar{\lambda})}^{G/M} \Def_{G/M}^G = e_{(LM,\lambda')}^G,$

\smallskip

\noindent (d) $e_{(L, \lambda)}^G 
        = \Inf_{\bar{G}}^G \, e_{(\tilde{L}, {\tilde\lambda})}^{\bar{G}} \, \Def_{\bar{G}}^G,
    $
    where 
    \( \bar{G} := G / \ker(\lambda) \), and \( \tilde{L} := L / \ker(\lambda) \).

\noindent Finally if $L\cap M \not \leq \ker(\lambda)$ then the left hand side products in (a) and (b) above are equal to zero.
\end{lem}

\begin{proof}(a) By definition the left hand side of the equation is equal to
        \begin{eqnarray*}
            e_{(L, \lambda)}^G \Inf_{G/M}^G = \Big [ \frac{G\times G}{\Delta_L(G),\phi_{\lambda}} \Big ]\Big [ \frac{G\times G/M}{\Delta_{\pi}(G),1} \Big ] = \Big [ \frac{G\times G/M}{\Delta_L(G) * \Delta_{\pi}(G), \phi_{\lambda} * 1}\Big],
        \end{eqnarray*}
        where $\pi : G \rightarrow G/M$ is the canonical epimorphism. The product is non-zero since $\lambda\mid_{L\cap M} = 1\mid_{L\cap M}$ by the condition  that $L\cap M\leq \ker(\lambda)$. Then it is straightforward to prove that 
        \[
        \Delta_L(G) * \Delta_{\pi}(G) = \Delta_{\pi}(G) *\Delta_{LM/M}(G/M).
        \]
        Moreover, for $\widetilde{\phi_{\bar{\lambda}}}\in \Big (\Delta_{LM/M}(G/M) \Big)^*$ we have
        \begin{eqnarray*}
            (1*\widetilde{{{\phi_{\bar{\lambda}}}}}) (g,hM) &=& \widetilde{\phi_{\bar{\lambda}}}(gM,hM) = \bar{\lambda} (h^{-1}gM)= \\
            &=& \lambda(h^{-1}g) = \phi_{\lambda}(g,h) = (\phi_{\lambda}*1)(g,hM).
        \end{eqnarray*}
        Hence,
        \begin{eqnarray*}
            \Inf_{G/M}^G e_{LM/M,\bar{\lambda}}^{G/M} = \Big [ \frac{G\times G/M}{\Delta_{\pi}(G),1} \Big ]  \Big [ \frac{G/M \times G/M}{\Delta_{LM/M}(G/M),\widetilde{\phi_{\bar{\lambda}}}} \Big ] = \Big [ \frac{G\times G/M}{\Delta_L(G)*\Delta_{\pi}(G),\phi_{\lambda}}\Big].
        \end{eqnarray*}
       
        \noindent Part (b) can be obtained by taking the opposites of both sides of the equation in (a). Part (c) follows by composing the equation in part (a) from right with $\Def_{G/M}^G$. 
        Part~(d) follows by applying the decomposition given in Equation (\ref{basiselementdecomposition}).

        Finally the products in parts (a), (b) are being zero when $L\cap M \not \leq \ker(\lambda)$ follows from the Mackey product formula for fibered bisets.
\end{proof}

\begin{nothing} We now define a finer system of idempotents in $\mathcal E(G)$, parametrized by pairs 
$(K,\kappa)$ with $K\le\Phi(G)$ normal in $G$ and $\kappa\in K^*$ $G$-invariant. We write \(\calM_{\Phi(G)}^G\) for the set of all such pairs. Note that $\calM_{\Phi(G)}^{G}$ is a subset of $\calM_G^G$ and inherits its poset structure from the latter. Below we also consider the subset of such pairs in which $\kappa$ is faithful; we denote this subposet of $\calM_{\Phi(G)}^{G}$ by $\calM_{G}^{*}$. 

Given $(K,\kappa)\in \calM_{\Phi(G)}^G$, set
    \begin{equation}
    \label{phiKkappa}
            \varphi_{K,\kappa}^G := \sum_{(K,\kappa)\leq (L,\lambda)\in \calM_{\Phi(G)}^G}\mu_{\unlhd G}(K,L) \tildeeLl.
    \end{equation}
\end{nothing}
These will later yield an orthogonal decomposition of 
$\mathcal E(G)$ and of fibered $p$-biset functor evaluations.
With the next proposition we may always replace $\kappa$ with a faithful character. To have a comparison, we note that our definition is inspired by \cite[Notation 3.4]{double} where Bouc defines
    \begin{equation}
    \label{Boucsphi}
            \varphi_{K}^G := \sum_{\substack{L\unlhd G\\ K\leq L\le \Phi(G)}}\mu_{\unlhd G}(K,L)\ Y_{L}.
    \end{equation}
where $Y_L = Y_{L, 1}$ in our notation.
\begin{prop}
\label{infphidef}
Let \( N \trianglelefteq G \) with \( N \leq \Phi(G) \) and let \( (K, \kappa) \in \calM_{\Phi(H)}^H \) where \( H := G / N \). Then 
\[
    \Inf_{H}^G \, \varphi_{K, \kappa}^H \, \Def_{H}^G = \varphi_{K', \kappa'}^G,
\]
where \( K' \leq G \) is the full preimage of \( K \) in \(G\), and \( \kappa' \) is  the pullback of \(\kappa\).
\end{prop}

\begin{proof} 
    For any summand $\tildeeLl$ of $\varphi_{K,\kappa}^H$, we have \( K \leq L \unlhd H \) and $L\leq \Phi(H)$. Additionally, by the correspondence theorem, \( \mu_{\unlhd H}(K,L) = \mu_{\unlhd G}(K', L') \). Now, by using the relations in \ref{elemopontilde}(b) together with Lemma \ref{infdefonipot}, we show that each summand on the left-hand side corresponds uniquely to a summand on the right-hand side; hence, the two sums are equal.
    \begin{eqnarray*}
        \Inf_{H}^G \widetilde{e_{L,\lambda}^H}\Def_{H}^G  &=&  \Inf_{H}^G\cdot \widetilde{e_{H}^{H}} \cdot  e_{(L,\lambda)}^H\cdot  \widetilde{e_{H}^{H}} \cdot \Def_{H}^G \\
        &=& \widetilde{e_G^G}\cdot  \Inf_{H}^G \cdot e_{(L,\lambda)}^H\cdot  \Def_{H}^G \cdot \widetilde{e_G^G} \\
        &=& \widetilde{e_G^G}\cdot  e_{(L',\lambda')}^G \cdot \widetilde{e_G^G} \\
    \end{eqnarray*}
    where $K'\leq L'$ and $\lambda'\mid_{K'}=\kappa'$. Note that $K'\leq \Phi(G)$ as $N$ is also a subgroup of $\Phi(G)$.
\end{proof}
\begin{remarki}
    One may compare $\varphi_{K, \kappa}^G$ with the idempotent $f_{(K,\kappa)}$ in \cite{fibered biset}. Clearly $\varphi_{K,\kappa}^G$ is a summand of the latter one. The restriction on the terms is what we need to obtain a block decomposition of the fibered $p$-biset category.
\end{remarki}
\begin{prop}
    \label{phikappaiso}
    Let \( G \) and \( H \) be finite groups. For a given pair \( (K, \kappa) \in \calM_{\Phi(G)}^G \), we have
    \begin{enumerate}[(a)]
    \item  \begin{align*}
        \varphi_{K,\kappa}^G = \widetilde{e_G^G}\left(\sum_{(K,\kappa)\leq  (L,\lambda)\in \calM_{\Phi(G)}^G} \mu_{\unlhd G} (K,L) e_{(L,\lambda)}^G\right) \\
        = \left(\sum_{(K,\kappa)\leq  (L,\lambda)\in \calM_{\Phi(G)}^G} \mu_{\unlhd G} (K,L) e_{(L,\lambda)}^G\right)\widetilde{e_G^G}.
    \end{align*}
    \item In particular
    \begin{equation}
        \label{varphiexplicitly}
        \varphi_{K,\kappa}^G = \frac{1}{|G|}\sum_{\substack{\Phi(G) \leq X\leq G\\ (K,\kappa)\leq (L,\lambda)\in \calM_{\Phi(G)}^G}} |X| \mu(X,G)\mu_{\unlhd G} (K,L) \Ind_X^G e_{(L,\lambda)}^X \Res_X^G.
    \end{equation}
    Here $e_{(L,\lambda)}^X$ denotes the idempotent $\Res^G_Xe_{(L,\lambda)}^G\Ind_X^G$.
    \item Let $f:G\rightarrow H$ be a group isomorphism. Then
    \[
     \Iso(f) \varphi_{K,\kappa}^G  = \varphi_{f(K),\kappa \circ f^{-1} }^H \Iso(f). 
    \]
    \item Let $p$ be a prime number. Suppose that $P$ is a finite $p$-group and $(K, \kappa)\in\calM_P^*$. Then $K$ is a cyclic central subgroup of $P$ and 
    \[
        \varphi_{K,\kappa}^P = \widetilde{e_P^P} \left( \sum_{\substack{(K,\kappa) \leq (L,\lambda) \in \calM_{\Phi(P)}^{P} \\ L/K \leq \Omega_1 Z(P/K)}} (-1)^{r_L} p^{\binom{r_L}{2}}\, e_{(L,\lambda)}^P \right),
    \]
    where $\Omega_1 Z(P/K)$ denotes the largest elementary abelian central subgroup of $P/K$, and $r_L =\text{rank}(L/K)$.
    \end{enumerate}
\end{prop}

\begin{proof} (a) Follows directly from Lemma \ref{etilde}.

\smallskip

\noindent (b)
Using the expansion of $\widetilde{e_G^G}$ in Section~\ref{section:eGG} we write the equation in part (a) as
\[
\varphi_{K,\kappa}^G = \frac{1}{|G|}\sum_{\substack{X\leq G\\ (K,\kappa)\leq (L,\lambda)\in \calM_{\Phi(G)}^G}} |X| \mu(X,G)\mu_{\unlhd G} (K,L) \Bigl[ \frac{G\times G}{\Delta(X),1} \Bigr] e_{(L,\lambda)}^G.
\]
By Theorem 2.3 of~\cite{H} \(\mu(X,G)=0\) unless \(\Phi(G)\leq X\). Hence, we only consider terms such that \(L \leq \Phi(G) \leq X\) for which we have \(\Delta(X) * \Delta_L(G) = \Delta_L(X).\) Thus,
\[
\Bigl[ \frac{G\times G}{\Delta(X),1} \Bigr] e_{(L,\lambda)}^G 
= \Bigl[ \frac{G\times G}{\Delta_L(X), \phi_{\lambda}} \Bigr]
= \Ind_X^G e_{(L,\lambda)}^X \Res_X^G,
\]
thereby completing the proof.
    
\smallskip

\noindent (c)  To prove this assertion we calculate
    \[
        \Iso(f)\varphi_{K,\kappa}^G \Iso(f^{-1}).
    \]
    Expanding the product and the following calculations give the result as follows.
    \begin{eqnarray*}
        & &\Iso(f)\varphi_{K,\kappa}^G \Iso(f^{-1})=\\
        &=&\Iso(f)\widetilde{e_G^G}\Iso(f^{-1})\Iso(f) \sum_{(K,\kappa)\leq  (L,\lambda)\in \calM_{\Phi(G)}^G} \mu_{\unlhd G}(K,L) e_{(L,\lambda)}^G \Iso(f^{-1}) \\
        &=& \widetilde{e_H^H} \sum_{(K,\kappa)\leq  (L,\lambda)\in \calM_{\Phi(G)}^G} \mu_{\unlhd G}(K,L) \Iso(f) e_{(L,\lambda)}^G \Iso(f^{-1}) \\
        &=& \widetilde{e_H^H} \sum_{(K,\kappa)\leq  (L,\lambda)\in \calM_{\Phi(G)}^G} \mu_{\unlhd G}(K,L)e_{(f(L),\lambda \circ f^{-1})}^H \\
        &=& \widetilde{e_H^H} \sum_{(K,\kappa)\leq  (L,\lambda)\in \calM_{\Phi(G)}^G} \mu_{\unlhd H}(f(K),f(L))e_{(f(L),\lambda \circ f^{-1})}^H \\
        &=& \varphi_{f(K),\kappa \circ f^{-1}}^H
    \end{eqnarray*}
\smallskip
\noindent (d)  Note that since $\kappa \in K^*$ is a faithful character, $K$ is a central and cyclic subgroup of $P$. Moreover, we have
    \[
        \mu_{\unlhd P}(K,L) = \mu_{\unlhd P/K}(1,L/K).
    \]
    By the argument in \cite[Lemma 6.2.10]{biset functor}, it follows that $\mu_{\unlhd P/K}(1,L/K) = 0$ unless $L/K$ is a central elementary abelian subgroup of $P/K$. In this case, we also have
    \[
        \mu_{\unlhd P/K}(1,L/K) = \mu(1,L/K)= (-1)^r p^{\binom{r}{2}},
    \]
  where $r = \operatorname{rank}(L/K)$.
\end{proof}

\begin{coro}
\label{elementarytimesphi}
Let $G$ be a finite group, and $(K,\kappa)\in \calM_{\Phi(G)}^G$. Then
    \begin{enumerate} [a)]
        \item $\Res_H^G \varphi_{K,\kappa}^G = 0$ and $\varphi_{K,\kappa}^G \Ind_H^G = 0$ for all proper subgroups $H$ of $G$.
        \item $\Def_{G/M}^G \varphi_{K,\kappa}^G = 0$ and $\varphi_{K,\kappa}^G\Inf_{G/M}^G = 0$ unless $M\cap \Phi(G) \leq K$ and $M\cap K\le \ker \kappa$.
    \end{enumerate}
\end{coro}
\begin{proof} We follow Bouc's proof of \cite[Corollary 3.7]{double}. (a) This follows directly from Section \ref{elemopontilde}.
        
\smallskip

\noindent (b) By transitivity of deflation, we can write:
        \[
        \Def_{G/M}^G = \Def_{G/M}^{G/(M\cap \Phi(G))} \Def_{G/(M\cap \Phi(G))}^G.
        \]
        Thus, it suffices to show that $\Def_{G/(M\cap \Phi(G))}^G \varphi_{K,\kappa}^G = 0$. Assuming now that $M \leq \Phi(G)$, we compute:
        \begin{eqnarray*}
            \Def_{G/M}^G \varphi_{K,\kappa}^G &=& \Def_{G/M}^G \left(\sum_{\substack{(K,\kappa)\leq (L,\lambda)\in \calM_{\Phi(G)}^G}} \mu_{\unlhd G} (K,L) e_{(L,\lambda)}^G  \right)\widetilde{e_G^G} \\
            &=& \left(\sum_{\substack{(K,\kappa)\leq (L,\lambda)\in \calM_{\Phi(G)}^G}} \mu_{\unlhd G} (K,L) \Big ( \Def_{G/M}^G e_{(L,\lambda)}^G \Big ) \right)\widetilde{e_G^G}.
        \end{eqnarray*}

        If $M\cap K \not\leq \widehat{K}$, then by the poset structure on $\calM_{\Phi(G)}^G$, we have $M\cap L \not\leq \widehat{L}$ for all $(L,\lambda) \geq (K,\kappa)$. Consequently, by Lemma \ref{infdefonipot} all terms in the sum vanish. Now suppose instead that $M\cap K \leq \widehat{K}$, in which case some terms might be nonzero. Then we rewrite the last equality using Lemma \ref{infdefonipot} as
        \begin{eqnarray*}
            & &\Def_{G/M}^G \varphi_{K,\kappa}^G=\\
            &=& \left(\sum_{\substack{(K,\kappa)\leq (L,\lambda)\in \calM_{\Phi(G)}^{G}\\ L\cap M\leq \widehat{L}}} \mu_{\unlhd G} (K,L) e_{(LM/M,\bar{\lambda})}^{G/M}  \Def_{G/M}^G  \right)\widetilde{e_G^G} \\
            &=& \sum_{\substack{(KM,\kappa') \leq (T,\tau) \in \calM_{\Phi(G)}^G\\ M\leq \widehat{T}}}
            \left(\sum_{\substack{(K,\kappa)\leq (L,\lambda)\in \calM_{\Phi(G)}^G\\ LM=T}} \mu_{\unlhd G} (K,L) \right) e_{(T/M,\bar{\tau})}^{G/M}\Def_{G/M}^G\widetilde{e_G^G}.
        \end{eqnarray*}
        Here, we define $\kappa' \colon KM \to \CC^{\times}$ by $\kappa'(km) = \kappa(k)$. By construction, we also have $\tau|_L = \lambda$, ensuring that for each $L \geq K$ such that $LM = T$, there exists a unique $\lambda$. Hence the inner sum runs only through the normal subgroups of $G$ over $K$. Now we apply~\cite[Proposition 4]{R} to the lattice of normal subgroups of $G$ over $K$ with the closure operation mapping $L$ to $LM$ for any $K\leq L \unlhd G$. It gives that the coefficient of $e_{(T/M,\bar{\tau})}^{G/M}$ in the above expression is zero unless $KM = K$, which is equivalent to $M \leq K$. This concludes the proof.
\end{proof}

\begin{theorem}
\label{phiortogonal} 
The set of elements $\{\varphi_{K, \kappa}^G \}$, indexed by $\calM_{\Phi(G)}^{G}$, forms a complete set of orthogonal idempotents in the algebra $\calE(G)$. That is, 

\begin{enumerate}
    \item   $\varphi_{K,\kappa}^G \cdot \varphi_{L,\lambda}^G = \delta_{(K,\kappa), (L,\lambda)}
                \varphi_{K,\kappa}^G.$
    \item $\sum_{(K,\kappa)}\varphi_{K,\kappa}^G = \tildeeG.$
\end{enumerate} 
\end{theorem}

\begin{proof} 
Utilizing the results established in Proposition~\ref{eKtimeseL} and Lemma~\ref{etilde}, we evaluate the following expression:
        \begin{equation} \label{Ydeltaproduct}
            \tildeeKk\cdot \widetilde{e_{(L,\lambda)}^G} = \widetilde{e_{G}^G} \cdot e_{(K,\kappa)}^G \cdot e_{(L,\lambda)}^G \cdot \widetilde{e_G^G} = \begin{cases}
               \widetilde{e_{(KL,\kappa\lambda)}^G}, & \text{if } \kappa\mid_{K\cap L} = \lambda\mid_{K\cap L},\\ 0, & \text{otherwise}.
            \end{cases}
        \end{equation}
       Then by Möbius inversion, it follows that
        \begin{equation} \label{Ydeltaassumphi}
            \tildeeKk = \sum_{(K,\kappa) \leq (L,\lambda) \in \calM_{\Phi(G)}^G} \varphi_{L,\lambda}^G.
        \end{equation}
        Additionally, by setting $(K,\kappa) = (1,1)$ in (\ref{Ydeltaassumphi}), we obtain
        \begin{equation*}
            \widetilde{e_G^G} = \sum_{(L,\lambda) \in \calM_{\Phi(G)}^G} \varphi_{L,\lambda}^G.
        \end{equation*}
        The proof proceeds by showing that for all $(K,\kappa), (L,\lambda) \in \calM_{\Phi(G)}^G$,
        \begin{equation}
        \label{Ytimesphi}
            \tildeeKk\cdot \varphi_{L,\lambda}^G = \varphi_{L,\lambda}^G \cdot \tildeeKk = \begin{cases}
                \varphi_{L,\lambda}^G, & \text{if } (K,\kappa) \leq (L,\lambda),\\
                0, & \text{otherwise},
            \end{cases}
        \end{equation}
        and
        \begin{equation*}
            \varphi_{K,\kappa}^G \cdot \varphi_{L,\lambda}^G = \begin{cases}
                \varphi_{K,\kappa}^G, & \text{if } (K,\kappa) = (L,\lambda),\\
                0, & \text{otherwise}.
            \end{cases}
        \end{equation*}
        The details follow arguments similar to those in \cite[Proposition 4.4]{fibered biset}. First, observe that if $(K,\kappa) \leq (L,\lambda)$, then
\begin{eqnarray*}
    \tildeeKk \cdot \varphi_{L,\lambda}^G &=& \tildeeKk \left(\sum_{(L,\lambda)\leq (N,\eta)\in \calM_{\Phi(G)}^G} \mu_{\unlhd G}(L,N) \widetilde{e_{(N,\eta)}^G}\right)\\
    &=& \sum_{(L,\lambda)\leq (N,\eta)\in \calM_{\Phi(G)}^G} \mu_{\unlhd G}(L,N) \tildeeKk \cdot \widetilde{e_{(N,\eta)}^G}\\
    &=& \sum_{(L,\lambda)\leq (N,\eta)\in \calM_{\Phi(G)}^G} \mu_{\unlhd G}(L,N) \widetilde{e_{(N,\eta)}^G} = \varphi_{L,\lambda}^G
\end{eqnarray*}
by equation (\ref{Ydeltaproduct}). Additionally, note that $\tildeeKk$ and $\varphi_{L,\lambda}^G$ commute since $\tildeeKk$ and $\widetilde{e_{(N,\eta)}^G}$ commute for all $(N,\eta) \geq (L,\lambda)$. Next, we establish the remaining statement of the proposition using induction on $d := d(K,\kappa) + d(L,\lambda)$, where $d(K,\kappa)$ is defined as the maximum $n \in \mathbb{N}_0$ such that there exists a chain $(K,\kappa) = (K_0,\kappa_0) < \cdots < (K_n,\kappa_n)$ in $\calM_{\Phi(G)}^G$. If $d = 0$, then $(K,\kappa)$ and $(L,\lambda)$ are maximal in $\calM_{\Phi(G)}^G$. Consequently, we have $\varphi_{K,\kappa}^G = \tildeeKk$ and $\varphi_{L,\lambda}^G = \widetilde{e_{(L,\lambda)}^G}$. It remains to demonstrate that $\tildeeKk \cdot \widetilde{e_{(L,\lambda)}^G} = 0$ for $(K,\kappa) \neq (L,\lambda)$. Assuming otherwise, equation (\ref{Ydeltaproduct}) implies that $\kappa\mid_{K\cap L} = \lambda\mid_{K\cap L}$ and that there exists a pair $(KL,\kappa\lambda) \in \calM_{\Phi(G)}^G$ with $(K,\kappa) \leq (KL,\kappa\lambda) \geq (L,\lambda)$. Given that $(K,\kappa)$ and $(L,\lambda)$ are maximal in $\calM_{\Phi(G)}^G$, we deduce $(K,\kappa) = (KL,\kappa\lambda) = (L,\lambda)$. Now, suppose $d > 0$ and that the proposition holds for smaller values of $d$. We first establish that if $\varphi_{K,\kappa}^G \varphi_{L,\lambda}^G \neq 0$, then $(K,\kappa) = (L,\lambda)$. Indeed, using the definition of the idempotents $\varphi_{K,\kappa}^G$ and $\varphi_{L,\lambda}^G$ given by Equation (\ref{phiKkappa}) and applying Equation (\ref{Ydeltaproduct}) yields:
\begin{equation*}
    \varphi_{K,\kappa}^G \cdot \varphi_{L,\lambda}^G \in \text{span}_k\{\widetilde{e_{(M,\mu)}^G}\mid (KL,\kappa\lambda) \leq (M,\mu) \in \calM_{\Phi(G)}^G\}\subseteq \mathcal E_G.
\end{equation*}
This implies that $\widetilde{e^G_{(KL,\kappa\lambda)}}\cdot \varphi_{K,\kappa}^G\cdot \varphi_{L,\lambda}^G = \varphi_{K,\kappa}^G\cdot \varphi_{L,\lambda}^G$ by equation (\ref{Ydeltaproduct}). Assuming $(K,\kappa) \neq (L,\lambda)$, we have $(K,\kappa) < (K L, \kappa \lambda)$ or $(L,\lambda) < (K L, \kappa \lambda)$. By induction, we obtain either $\widetilde{e_{(KL,\kappa\lambda)}^G}\cdot \varphi_{K,\kappa}^G = 0$ or $\widetilde{e_{(KL,\kappa\lambda)}^G}\cdot \varphi_{L,\lambda}^G = 0$. In either case, we reach a contradiction, concluding that $(K,\kappa) = (L,\lambda)$. Next, assuming $(K,\kappa) = (L,\lambda)$, we use Equations (\ref{Ydeltaproduct}) and (\ref{Ydeltaassumphi}) to conclude:
\begin{equation*}
    \tildeeKk = (\tildeeKk)^2 = \left( \sum_{(K,\kappa) \leq (K',\kappa') \in \calM_{\Phi(G)}^G} \varphi_{K',\kappa'}^G \right)^2.
\end{equation*}
Expanding this sum and comparing it with equation (\ref{Ydeltaassumphi}) yields $\varphi_{K,\kappa}^G \cdot\varphi_{K,\kappa}^G = \varphi_{K,\kappa}^G$, proving idempotency. Finally, from equation (\ref{Ydeltaassumphi}), we derive:
\begin{equation*}
    \tildeeKk\cdot \varphi_{L,\lambda}^G = \sum_{(K,\kappa) \leq (K',\kappa') \in \calM_{\Phi(G)}^G} \varphi_{K',\kappa'}^G\cdot \varphi_{L,\lambda}^G.
\end{equation*}
Using induction and previously established results, the sum evaluates to $\varphi_{L,\lambda}^G$ if $(K,\kappa) \leq (L,\lambda)$ and 0 otherwise.
\end{proof}

\begin{theorem}\label{phiortogonal2}
Let $(K,\kappa)\in \calM_{\Phi(G)}^G$ and $H$ be a finite group.
        \begin{enumerate}[a)]
            \item If $(U,\upsilon) \in \calM_{G\times H}$, then $\varphi_{K,\kappa}^G \Big[ \frac{G\times H}{U,\upsilon}\Big] = 0$ unless $p_1(U) = G$ and $k_1(U)\cap \Phi(G) \leq K$ and $\kappa\mid_{k_1(U)\cap K}=\upsilon_1\mid_{k_1(U)\cap K}$.
            \item If $(U',\upsilon') \in \calM_{H\times G}$, then $\Big[ \frac{H\times G}{U',\upsilon'}\Big] \varphi_{K,\kappa}^G = 0$ unless $p_2(U') = G$ and $k_2(U')\cap \Phi(G) \leq K$ and $\kappa\mid_{k_2(U')\cap K}=\upsilon'_2\mid_{k_2(U')\cap K}$.
        \end{enumerate}

\end{theorem}

\begin{proof} We prove part (a).  Part (b) follows similarly. First, we expand the product as
\begin{eqnarray*}
\varphi_{K,\kappa}^G\left[\frac{G \times H}{U, \upsilon}\right]
& = &
\widetilde{e_G^G} \left( \sum_{(K, \kappa) \leq (L, \lambda) \in \calM_{\Phi(G)}^{G}} \mu_{\unlhd G}(K, L) \, e_{(L, \lambda)}^G \right) \left[\frac{G \times H}{U, \upsilon}\right] \\
& = &
\widetilde{e_G^G} \left( \sum_{\substack{(K, \kappa) \leq (L, \lambda) \in \calM_{\Phi(G)}^{G} \\
\lambda|_{L \cap p_1(U)} = \upsilon_1|_{L \cap p_1(U)}}}
\mu_{\unlhd G}(K, L) \left[\frac{G \times H}{(L \times 1)U, \lambda \cdot \upsilon}\right] \right),
\end{eqnarray*}
where \((\lambda \cdot \upsilon)((l,1)u) := \lambda(l)\, \upsilon(u)\). Let \(M := k_1(U) \cap \Phi(G)\). We claim that
\[
(L \times 1)U = (LM \times 1)U.
\]

The inclusion \((L \times 1)U \subseteq (LM \times 1)U\) is immediate. For the reverse inclusion, let \(x \in (LM \times 1)U\), so that \(x = (lm, 1)(u_1, u_2)\) for some \(l \in L\), \(m \in M\), and \((u_1, u_2) \in U\). Since \(m \in M = k_1(U) \cap \Phi(G)\), we have \((m, 1) \in U\). Hence,
\[
x = (l, 1)(m, 1)(u_1, u_2) = (l, 1)(mu_1, u_2) \in (L \times 1)U,
\]
which establishes the claimed equality.

Assume now that the above sum is nonzero. Then there exists an element \((T, \tau)\) such that the coefficient of the basis element corresponding to \(((T \times 1)U, \tau \cdot \upsilon)\) in the sum is nonzero. Moreover, \(((T \times 1)U, \tau \cdot \upsilon)\) is of the form \(((LM \times 1)U, \lambda \cdot \upsilon)\) for some \((K, \kappa) \leq (L, \lambda) \in \calM_{\Phi(G)}^{G}\). In particular, the coefficient of the corresponding element
\[
\left[\frac{G \times H}{(T \times 1)U, \tau \cdot \upsilon}\right]
\]
in the sum is given by
\[
\sum_{\substack{(K, \kappa) \leq (L, \lambda) \in \calM_{\Phi(G)}^{G} \\ LM = T}} \mu_{\unlhd G}(K, L).
\]
Note that if \(K \leq L \unlhd \Phi(G)\) and \((LM \times 1)U = (T \times 1)U\), then there exists a unique \(\lambda \in L^*\) such that
\[
((LM \times 1)U, \lambda \cdot \upsilon) = ((T \times 1)U, \tau \cdot \upsilon),
\]
since \((\tau \cdot \upsilon)|_L = \tau|_L = \lambda\). Since we assume that the above coefficient is nonzero, it follows from~\cite[Proposition 4]{R} that \(KM \leq K\), which in turn implies that \(M \leq K\). 

Finally, the condition $\kappa\mid_{k_1(U)\cap K}=\upsilon_1\mid_{k_1(U)\cap K}$ arises from the Mackey formula for the product of fibered bisets. To establish the result, it suffices to show that if
\begin{equation*}
    \kappa\mid_{k_1(U)\cap K} \neq \upsilon_1\mid_{k_1(U)\cap K},
\end{equation*}
then for any pair $(K,\kappa) \leq (L,\lambda) \in \calM_{\Phi(G)}^G$, we have
\begin{equation*}
    \lambda\mid_{k_1(U)\cap L} \neq \upsilon_1\mid_{k_1(U)\cap L}.
\end{equation*}
Consequently, all terms in the sum vanish, implying that the product is zero. To see why this holds, assume there exists a pair $(L,\lambda)$ such that
\begin{equation*}
    \lambda\mid_{k_1(U)\cap L} = \upsilon_1\mid_{k_1(U)\cap L}.
\end{equation*}
By the poset structure, it follows that
\begin{equation*}
    \lambda\mid_{k_1(U)\cap K} = \kappa\mid_{k_1(U)\cap K} \neq \upsilon_1\mid_{k_1(U)\cap K}.
\end{equation*}
However, since $k_1(U) \cap K \leq k_1(U) \cap L$, we must have
\begin{equation*}
    \lambda\mid_{k_1(U)\cap K} = \upsilon_1\mid_{k_1(U)\cap K},
\end{equation*}
which contradicts our assumption. This completes the proof of part (a).
\end{proof}
\begin{remarki}
  The conditions in Theorem \ref{phiortogonal2} imply in particular that \( \widehat{k_1(U)} \cap \Phi(G) \leq \widehat{K} \) and \(\widehat{k_1(U)} \cap K \leq \widehat{K}\), where hats denote the kernels of the respective characters. This will be useful in verifying vanishing or orthogonality conditions later.
\end{remarki}
\begin{prop}
    Let $G$ be a group with order invertible in $k$ and let $(K,\kappa)\in \calM_{\Phi(G)}^G$ and 
    let $(U,\upsilon)\in \calM_{G\times G}$ such that $U$ is a covering subgroup of $G\times G$. Then
        \begin{eqnarray*}
            \varphi_{K,\kappa}^G\cdot Y_{U,\upsilon} = \sum_{\substack{(K,\kappa) \leq (L,\lambda)\in \calM_{\Phi(G)}^G\\ \lambda\mid_{L\cap k_1(U)} = \upsilon_1\mid_{L\cap k_1(U)} }} \mu_{\unlhd G}(K,L)Y_{(L\times 1)U,\lambda\cdot \upsilon}.
        \end{eqnarray*}
        This product is non-zero if and only if $k_1(U)\cap \Phi(G) =K$ and $\kappa\mid_{K\cap k_1(U)} = \upsilon_1\mid_{K\cap k_1(U)}$. Similarly for the left multiplication, we have
    \begin{equation*}
             Y_{U,\upsilon}\cdot \varphi_{K,\kappa}^G = \sum_{\substack{(K,\kappa) \leq (L,\lambda)\in \calM_{\Phi(G)}^G\\ \lambda\mid_{L\cap k_2(U)} = \upsilon_2\mid_{L\cap k_2(U)} }} \mu_{\unlhd G}(K,L)Y_{(1\times L)U,\lambda\cdot \upsilon},
       \end{equation*}
        and $Y_{U,\upsilon}\cdot \varphi_{K,\kappa}^G \not = 0$ if and only if $k_2(U)\cap \Phi(G) = K$ and $\upsilon_2\mid_{K\cap k_2(U)} = \kappa\mid_{K\cap k_2(U)}$.
\end{prop}
\begin{proof} 
By definition of \(Y_{U,\upsilon}\), the product \(\varphi_{K,\kappa}^G \cdot Y_{U,\upsilon}\) is equal to
\[
\varphi_{K,\kappa}^G \cdot \widetilde{e_G^G}\Bigl[\frac{G\times G}{U,\upsilon}\Bigr] \widetilde{e_G^G}.
\]
Following Proposition~\ref{phikappaiso}(a), the idempotent \(\widetilde{e_G^G}\) on the left is absorbed by \(\varphi_{K,\kappa}^G\), and hence this expression reduces to
\[
\varphi_{K,\kappa}^G \cdot \Bigl[\frac{G\times G}{U,\upsilon}\Bigr] \widetilde{e_G^G}.
\]
Finally, using the computations from the proof of Theorem~\ref{phiortogonal2}, we obtain the first statement of the theorem. The product being nonzero implies the given conditions by Theorem~\ref{phiortogonal2}. To demonstrate that they are sufficient for the product to be nonzero, we examine the coefficient of \( Y_{U,\upsilon} \) in the sum. Observe that \( Y_{U,\upsilon} = Y_{(L \times 1)U,\, \lambda \cdot \upsilon} \) if and only if \( (U, \upsilon) \) is conjugate to \( ((L \times 1)U,\, \lambda \cdot \upsilon) \). This implies that \( L \) is contained in the subgroup chain
    \[
        k_1((L \times \1)U) \leq k_1(U),
    \]
    leading to the relation
    \[
        K \leq L \leq k_1(U) \cap \Phi(G) = K,
    \]
    which forces \( L = K \). Consequently, the coefficient of \( Y_{U,\upsilon} \) in \( \varphi_{K,\kappa} Y_{U,\upsilon} \) is equal to $\1$, confirming that \( \varphi_{K,\kappa} \cdot Y_{U,\upsilon} \neq 0 \). 
    
    As Proposition~\ref{phikappaiso} (a) and Theorem~\ref{phiortogonal2} are valid for the right multiplication the rest can be established by a similar argument.
\end{proof}
\begin{coro}
    The elements $\varphi_{K,\kappa}^G Y_{U,\upsilon}$ 
    as $(U,\upsilon)$ ranges over a set of representatives of the conjugacy classes of elements in $\mathcal{M}_{G \times G}$ with $U$ covering, which satisfies the conditions
\[
k_1(U) \cap \Phi(G) = K \quad \text{and} \quad \kappa|_{K \cap k_1(U)} = \upsilon_1|_{K \cap k_1(U)}
\]
form a $k$-basis of the right ideal $\varphi_{K,\kappa}^G \mathcal{E}(G)$ 
of $\mathcal{E}(G)$. Similar basis exists for the corresponding left ideal.
\end{coro}
\begin{proof}
   Since the coefficient of \( Y_{U,\upsilon} \) in the product $\varphi_{K,\kappa}^G Y_{U,\upsilon}$ is equal to ${\bf 1}$  when \( k_1(U) \cap \Phi(G) = K \) and \( \kappa\big|_{K \cap k_1(U)} = \upsilon_1\big|_{K \cap k_1(U)} \), the set of elements \( \varphi_{K,\kappa}^G Y_{U,\upsilon} \) forms a \( k \)-basis of \( \varphi_{K,\kappa}^G \mathcal{E}(G) \), see Section~\ref{basis}. 
\end{proof}

\section{Idempotents of double $\CC^{\times}$-fibered Burnside algebras}
Throughout this section let $G$ be a finite group and let $k$ be  a commutative ring in which $|G|$ is invertible. In this section, we define a set of orthogonal idempotent elements in the $\CC^{\times}$-fibered Burnside algebra $B_k^{\CC^{\times}}(G,G)$ which sum to the identity element. Consequently, we demonstrate that the evaluation of a $\CC^{\times}$-fibered biset functor at $G$ decomposes as a direct sum of $k$-modules, indexed by the minimal sections $(T, S)$ of $G$ and by specific elements of the poset $\calM_{\Phi(T/S)}^{T/S}$. Here, a \emph{minimal} section $(T,S)$ of $G$ is a section with $S \unlhd \Phi(T)$.

\begin{defi} 
Let $(T,S)$ be a section of $G$ 
and let $(K,\kappa) \in \calM_{T/S}^{*}$. Define:
\begin{eqnarray*}
u_{T,S,K,\kappa}^{G} = \Indinf^G_{T/S} \varphi_{K,\kappa}^{T/S}\in B^{\CC^{\times}}_k(G,T/S),\\
v_{T,S,K,\kappa}^{G} = \varphi_{K,\kappa}^{T/S} \Defres^G_{T/S}\in B^{\CC^{\times}}_k(T/S,G).
\end{eqnarray*}
Note that 
\[
u_{T,S,K,\kappa}^{G} = (v_{T,S,K,\kappa}^{G})^{op}.
\]
These elements will be used as partial idempotents in 
$B_k^\ctimes(G,G)$, projecting onto fibered components associated to certain triples $(T/S,K,\kappa)$.
\end{defi}
\begin{nothing}
The group  \( \Aut(G) \) acts on 
\(\calM_{\Phi(G)}^{G} \) via
\[
\psi \cdot (K, \kappa) = (\psi(K), \kappa \circ \psi^{-1}),
\]
where \( \psi \in \Aut(G), (K, \kappa) \in \calM_{\Phi(G)}^{G} \). Since $\Inn(G)$ acts trivially on the elements of $\calM_{\Phi(G)}^G$, the above action descents to an action of $\Out(G)$. Let 
\([K,\kappa]_{\Out(G)}\) denote the corresponding orbit. Usually we omit the index $\Out(G)$ and simply write $[K, \kappa] := [K, \kappa]_{\Out(G)}$ whenever it is clear from the context. 
We denote a complete set of orbit representatives of $\calM_{\Phi(G)}^G$ by $[\calM_{\Phi(G)}^G]$. Similarly we have the subset $[\calM_G^*]$ of representatives for pairs with faithful characters.
\end{nothing}

\begin{nothing}
For $(K,\kappa)\in [\calM_{\Phi(G)}^G]$, set
\[
\varphi_{[K,\kappa]}^G := \sum_{(K',\kappa')\in [K,\kappa]} \varphi_{K',\kappa'}^G
\]
and similarly for a minimal section $(T, S)$ and  $(K,\kappa)\in [\calM_{\Phi(T/S)}^{T/S}]$, set
\[
u_{T,S,[K,\kappa]}^G := \sum_{(K',\kappa')\in [K,\kappa]} u_{T,S,K'\kappa'}^G
\quad \text{and}\quad
v_{T,S,[K,\kappa]}^G := \sum_{(K',\kappa')\in [K,\kappa]} v_{T,S,K'\kappa'}^G.
\]
The following theorem and its proof are the adaptations of \cite[Theorem 4.5]{double} and its proof to the fibered case.
\end{nothing}

\begin{theorem}
\label{uvorthogonal}
Let $(T, S), (T', S')$ be minimal sections of $G$. Also let $(K,\kappa)\in [\calM_{T/S}^{*}]$ and  $(K',\kappa')\in [\calM_{T'/S'}^{*}]$. Then
\begin{enumerate}[a)]
\item $v_{T',S',[K',\kappa']}^G \cdot u_{T,S,[K,\kappa]}^G =0$
unless $(T,S) =_G (T',S')$ and $[K,\kappa] = [K',\kappa']$.
\item When $(T, S) = (T', S')$ then
\begin{eqnarray*}
v_{T,S,[K,\kappa]}^{G}\cdot u_{T,S,[K,\kappa]}^{G} 
&=& \varphi_{[K,\kappa]}^{T/S} \left( \sum_{g \in N_G(T, S)/T} \text{Iso}(c_g) \right)\\
&=& \left( \sum_{g \in N_G(T, S)/T} \text{Iso}(c_g) \right) \varphi_{[K,\kappa]}^{T/S}.
\end{eqnarray*}
Here $N_G(T, S) = N_G(T) \cap N_G(S)$, and $c_g$ is the conjugation by $g$ on $T/S$.
\end{enumerate}
\end{theorem}

\begin{proof} (a) Suppose that $v_{T',S',[K',\kappa']}^G u_{T,S,[K,\kappa]}^G \neq 0$. Then there exist a pair $(K_2,\kappa_2) \in [K,\kappa]$ and $(K_1,\kappa_1) \in [K',\kappa']$ such that \(v_{T',S',K_1,\kappa_1}^G u_{T,S,K_2,\kappa_2}^G\not = 0\). By definition and by the Mackey formula we have the expression
\begin{align*}
v_{T',S',K_1,\kappa_1}^G u_{T,S,K_2,\kappa_2}^G &= \varphi_{K_1,\kappa_1}^{T'/S'} \Defres^G_{T'/S'} \Indinf^G_{T/S} \varphi_{K_2,\kappa_2}^{T/S} \\
&= \varphi_{K_1,\kappa_1}^{T'/S'} \left( \sum_{g\in T'\setminus G /T} \Big[\frac{T'/S'\times T/S}{U_g,1}\Big]\right) \varphi_{K_2,\kappa_2}^{T/S},
\end{align*}
where $U_g = \{(t'S',tS) \in (T'/S') \times (T/S) \mid t'gt^{-1} \in S'gS\}$. In particular, the projections satisfy $p_1(U_g) = ({}^gT \cap T')S'/S'$ and $p_2(U_g) = (T'^g \cap T)S/S$. Notably, $p_1(U_g) = T'/S'$ if and only if ${}^gT \cap T' = T'$, as $S' \leq \Phi(T')$. Similarly, $p_2(U_g) = T/S$ if and only if $T'^g \cap T = T$. By Theorem~\ref{phiortogonal2}, it follows that
\[
\varphi_{K_1,\kappa_1}^{T'/S'} \left( \sum_{g\in T'\setminus G /T} \Big[\frac{T'/S'\times T/S}{U_g,1}\Big]\right) \varphi_{K_2,\kappa_2}^{T/S} = 0,
\]
unless $T' = {}^gT$. Now assume that $T' = {}^gT$, for some $g\in G$. Then we have $k_1(U_g) = ({}^gS \cap T')S'/S' = {}^gSS'/S'$ and $k_2(U_g) = (S'^g \cap T)S/S = S'^gS/S$. Moreover, since ${}^gS \leq {}^g\Phi(T) = \Phi(T')$ and $S'^g \leq \Phi(T'^g) = \Phi(T)$, it follows that ${}^gSS'/S'$ is contained in $k_1(U_g) \cap (\Phi(T')/S') = k_1(U_g) \cap \Phi(T'/S')$. Similarly, $S'^gS/S$ is contained in $k_2(U_g) \cap (\Phi(T)/S) = k_2(U_g) \cap \Phi(T/S)$. By Theorem \ref{phiortogonal2}, the product
\[
\varphi_{K_1,\kappa_1}^{T'/S'} \left( \sum_{g\in T'\setminus G /T} \Big[\frac{(T'/S')\times (T/S)}{U_g,1}\Big]\right) \varphi_{K_2,\kappa_2}^{T/S}
\]
is equal to zero unless $k_1(U_g) \cap \Phi(T'/S') \leq K_1$, $\kappa_1\mid_{k_1(U_g)\cap K_1}=
1$
and $k_2(U_g) \cap \Phi(T/S) \leq K_2$ and $\kappa_2\mid_{k_2(U_g)\cap K_2}= 1$. But as the $\kappa_1$ and $\kappa_2$ are faithful characters we must have $k_1(U_g) \cap \Phi(T'/S')=k_2(U_g)\cap \Phi(T/S)=
1$. This implies that $S'^g \leq S \leq S'^g$, proving that the pairs $(T, S)$ and $(T', S')$ must be conjugate. Assuming that $(T,S)$ and $(T',S')$ are conjugate, the above product simplifies to
\[
\sum_{g\in N_G(T,S)/T} \left( \varphi_{K_1,\kappa_1}^{T/S} \Big[\frac{T/S\times T/S}{U_g,1}\Big] \varphi_{K_2,\kappa_2}^{T/S} \right) = \sum_{g\in N_G(T,S)/T} \left( \varphi_{K_1,\kappa_1}^{T/S} \Iso(c_g) \varphi_{K_2,\kappa_2}^{T/S} \right).
\]
By Proposition \ref{phikappaiso}, part (c), it follows that
\begin{align*}
\sum_{g\in N_G(T,S)/T} \left( \varphi_{K_1,\kappa_1}^{T/S} \Iso(c_g) \varphi_{K_2,\kappa_2}^{T/S} \right) &= \sum_{g\in N_G(T,S)/T} \left(\Iso(c_g) \varphi_{K_1^g,\kappa_1\circ c_{g^{-1}}}^{T/S} \varphi_{K_2,\kappa_2}^{T/S}\right)= 0,
\end{align*}
unless $(K_2,\kappa_2) = (K_1^g,\kappa_1\circ c_{g^{-1}})$ due to the orthogonality of these elements. Consequently, the proof is complete as it follows that $[K,\kappa] = [K_2,\kappa_2] = [K_1^g,\kappa_1\circ c_{g^{-1}}] = [K_1,\kappa_1]=[K',\kappa']$.

\noindent (b) From the calculations in the first part we can expand the product as
{\allowdisplaybreaks
\begin{eqnarray*}
& &v_{T,S,[K,\kappa]}^{G}\cdot u_{T,S,[K,\kappa]}^{G}=\\
&=& \left ( \sum_{(K',\kappa')\in [K,\kappa]} v_{T,S,K'\kappa'}^G \right ) \left( \sum_{(K'',\kappa'')\in [K,\kappa]} u_{T,S,K''\kappa''}^G\right)\\
&=&
\sum_{(K',\kappa')\in [K,\kappa]} \left( \sum_{(K'',\kappa'')\in [K,\kappa]} v_{T,S,K'\kappa'}^G \cdot u_{T,S,K''\kappa''}^G\right) \\ 
&=&
\sum_{\substack{(K',\kappa')\in [K,\kappa]\\ (K'',\kappa'')\in [K,\kappa]}} \left( \varphi_{K',\kappa'}^{T/S} \left( \sum_{g \in N_G(T, S)/T} \text{Iso}(c_g) \right)\varphi_{K'',\kappa''}^{T/S}\right)\\ 
&=&
\left( \sum_{(K',\kappa')\in [K,\kappa]} \varphi_{K',\kappa'}^{T/S} \right) \left( \sum_{g \in N_G(T, S)/T} \text{Iso}(c_g) \right) \left (\sum_{(K',\kappa')\in [K,\kappa]} \varphi_{K'',\kappa''}^{T/S}\right)\\ 
&=&
\left( \sum_{g \in N_G(T, S)/T} \text{Iso}(c_g) \right) \left( \sum_{(K',\kappa')\in [K,\kappa]} \varphi_{K',\kappa'}^{T/S} \right) \left (\sum_{(K',\kappa')\in [K,\kappa]} \varphi_{K'',\kappa''}^{T/S}\right)\\ 
&=&
\left( \sum_{g \in N_G(T, S)/T} \text{Iso}(c_g) \right) \left( \sum_{(K',\kappa')\in [K,\kappa]} \varphi_{K',\kappa'}^{T/S} \right) \\ &=& \left( \sum_{(K',\kappa')\in [K,\kappa]} \varphi_{K',\kappa'}^{T/S} \right)\left( \sum_{g \in N_G(T, S)/T} \text{Iso}(c_g) \right).
\end{eqnarray*}}
\end{proof}

\begin{defi}\label{defn:etskk}
Let $(T,S)$ be a minimal section of $G$ and $(K,\kappa)\in \calM_{T/S}^{*}$. Define
\begin{align*}
\epsilon_{T,S,[K,\kappa]}^G := \frac{1}{|N_G(T,S):T|}\cdot u_{T,S,[K,\kappa]}^G\cdot v_{T,S,[K,\kappa]}^G \in B_k^{\CC^{\times}}(G, G).
\end{align*}
These idempotents are orthogonal and sum to the identity, as we shall prove below. First expand as follows.
\begin{eqnarray*}
&&\epsilon_{T, S, [K, \kappa]}^G=\\
&=& \frac{1}{|N_G(T, S) : T|}
\cdot \left( \sum_{(K', \kappa')} \Indinf^G_{T/S} \varphi_{K', \kappa'}^{T/S} \right)
\left( \sum_{(K'', \kappa'')} \varphi_{K'', \kappa''}^{T/S} \Defres^G_{T/S} \right) \\
&=& \frac{1}{|N_G(T, S) : T|}
\cdot \sum_{\substack{(K', \kappa'), (K'', \kappa'') \in [K, \kappa]}}
\Indinf^G_{T/S} \varphi_{K', \kappa'}^{T/S}\cdot \varphi_{K'', \kappa''}^{T/S} \Defres^G_{T/S} \\
&=& \frac{1}{|N_G(T, S) : T|}
\cdot \sum_{(K', \kappa') \in [K, \kappa]}
\Indinf^G_{T/S} \varphi_{K', \kappa'}^{T/S} \Defres^G_{T/S} \\
&=& \frac{1}{|N_G(T, S) : T|}
\cdot \Indinf^G_{T/S} \varphi_{[K, \kappa]}^{T/S} \Defres^G_{T/S}.
\end{eqnarray*}
\end{defi}
Once more these idempotents are fibered versions of Bouc's idempotents $\epsilon_{T, S}^G$ given in \cite[Notation 4.6]{double}. The fibered version is obtained in the same way using $\varphi_{[K, \kappa]}^{T/S}$ instead of $\varphi_{\bf 1}^{T/S}$. 
\begin{theorem}
The set
\[
\{ \epsilon_{T,S,[K,\kappa]}^G \big| (T,S) \in [\Pi(G)], (K,\kappa) \in [\calM_{T/S}^* ]\}
\]
 is a system of orthogonal idempotents in $B_k^{\CC^{\times}}(G,G)$ summing up to the identity.
\end{theorem}
\begin{proof}
The orthogonality of the elements $\epsilon_{T,S,[K,\kappa]}^G$ follow simply by Theorem \ref{uvorthogonal}. Moreover,
\begin{eqnarray*}
&& \sum_{\substack{(T,S)\in [\Pi(G)]\\(K,\kappa)\in [\calM_{T/S}^{*}]}}\epsilon_{T,S,[K,\kappa]}^G=\\
&=&\sum_{\substack{(T,S)\in [\Pi(G)] \\(K,\kappa)\in [\calM_{T/S}^{*}]}} \frac{1}{|N_G(T,S):T|} \Indinf_{T/S}^G \varphi_{[K,\kappa]}^{T/S}\Defres_{T/S}^G\\
&=&\sum_{(T,S)\in \Pi(G)} \left(\sum_{\substack{(K,\kappa)\in \calM_{T/S}^{*}}} \frac{1}{|G:T|} \Indinf_{T/S}^G \varphi_{K,\kappa}^{T/S}\Defres_{T/S}^G\right).
\end{eqnarray*}
Fix $T\leq G$ and sum over $S$ to obtain
\begin{eqnarray*}
&&\sum_{\substack{S\unlhd T, S\leq \Phi(T)\\(K,\kappa)\in \calM_{T/S}^{*}}} \Indinf_{T/S}^G \varphi_{K,\kappa}^{T/S}\Defres_{T/S}^G\\
&=&\Ind_T^G\left(\sum_{\substack{S\unlhd T, S\leq \Phi(T)\\(K,\kappa)\in \calM_{T/S}^{*}}} \Inf_{T/S}^T \varphi_{K,\kappa}^{T/S}\Def_{T/S}^T\right) 
\Res_T^G.
\end{eqnarray*}
Observe that there exists a bijection between the set \(\calM_{\Phi(T)}^{T}\) and the set of triples \((S, K, \kappa)\) satisfying the following properties:
\[
S \unlhd T,\quad S \subseteq \Phi(T), \quad \text{and} \quad (K, \kappa) \in \calM_{T/S}^{*}.
\]
This correspondence is given by mapping each \((L, \lambda) \in \calM_{\Phi(T)}^{T}\) to the triple \((\widehat{L}, L / \widehat{L}, \bar{\lambda})\), where \(\widehat{L} = \ker(\lambda)\), and \(\bar{\lambda}(l\widehat{L}) = \lambda(l)\) for all \(l \in L\). Hence by Proposition \ref{infphidef} and also using Theorem \ref{phiortogonal} and \cite[Lemma 2.13]{double}, the above equation becomes
\begin{align*}
\Ind_T^G \left(\sum_{\substack{(L,\lambda)\in \calM_{\Phi(T)}^{T}}} \varphi_{L,\lambda}^{T}\right)\Res_T^G= \Ind_T^G \widetilde{e_T^T} \Res_T^G = |N_G(T):T|\widetilde{e_T^G}.
\end{align*}
Therefore we get
\begin{align*}
\sum_{\substack{(T,S)\in [\Pi(G)]\\(K,\kappa)\in [\calM_{T/S}^{*}]}}\epsilon_{T,S,[K,\kappa]}^G 
=
\sum_{T\leq G} \frac{|N_G(T):T|}{|G:T|} \widetilde{e_T^G} = \sum_{T\le_G G} \widetilde{e_T^G} = \Big [ \frac{G\times G}{\Delta(G),1} \Big].
\end{align*}
\end{proof}

\begin{nothing}
Having constructed a complete set of mutually orthogonal idempotents in \(\CC^\times\)-fibered Burnside algebra $B_k^{\CC^{\times}}(G,G)$, we are now equipped to describe a powerful structural decomposition of the evaluations of $\CC^{\times}$-fibered biset functors on finite groups. Let us consider such a functor \(F\) defined over a field \(k\), and fix a finite group \(G\) for which \(|G| \in k^\times\).

The fundamental result presented here is a canonical decomposition of the module \(F(G)\), which reveals its internal structure in terms of simpler, well-understood constituents. Specifically, \(F(G)\) can be expressed as a direct sum of components indexed by conjugacy classes of minimal sections \((T,S)\) of \(G\), and by equivalence classes \([K,\kappa]\), where \(K\) is a normal subgroup of the quotient \(T/S\) contained in its Frattini subgroup and \(\kappa\) is a faithful character of \(K\). This decomposition is realized via the isomorphism
\[
F(G) \cong \bigoplus_{(T,S) \in [\Pi(G)]} \bigoplus_{(K,\kappa) \in [\mathcal{M}_{T/S}^*]} \left( \varphi_{[K,\kappa]}^{T/S} F(T/S) \right)^{N_G(T,S)/T}.
\]

This decomposition is made explicit by two mutually inverse maps. The forward map, denoted \(V\), sends an element \(w \in F(G)\) to the tuple of its components in the above direct sum:
\[
V(w) = \left( \frac{1}{|N_G(T,S) : T|} \cdot v_{T,S,[K,\kappa]}^G \cdot w \right)_{(T,S,[K,\kappa])}.
\]
Each of these components belongs to the space \(\left( \varphi_{[K,\kappa]}^{T/S} F(T/S) \right)^{N_G(T,S)/T}\), and we must verify this invariance. To do so, fix \((T,S,[K,\kappa])\) and consider any \(g \in N_G(T,S)\). Since \(g\) normalizes both \(T\) and \(S\), the conjugation map \(c_g\) fixes the section \((T,S)\) and the orbit \([K,\kappa]\). Thus conjugation by \(g\) maps the idempotent \(v_{T,S,[K,\kappa]}^G\) to $\Iso(c_g)(v_{T,S,[K,\kappa]}^G) = v_{T, S, [K,\kappa]}^G \Iso(c_g)$. Moreover, the map \( \Iso(c_g):F(G)\to F(G) \) is an inner automorphism hence the identity map for $g\in G$ and we conclude that \(v_{T,S,[K,\kappa]}^G \cdot w\) is indeed fixed by the action of \(N_G(T,S)/T\).
The inverse map \(U\) reconstructs an element of \(F(G)\) from its projections in the summands:
\[
U\left( (w_{T,S, (K,\kappa)})_{(T,S, (K,\kappa))} \right) = \sum_{(T,S),[K,\kappa]} u_{T,S,[K,\kappa]}^G \cdot w_{T,S,[K,\kappa]}.
\]

We now verify that \(U \circ V = \mathrm{id}_{F(G)}\). For any \(w \in F(G)\), we compute
\[
U(V(w)) = \sum_{(T,S),[K,\kappa]} \frac{1}{|N_G(T,S) : T|} \cdot u_{T,S,[K,\kappa]}^G \cdot v_{T,S,[K,\kappa]}^G \cdot w.
\]
The sum over all \(u \cdot v\) terms recovers the identity map on \(F(G)\), since the elements \(\epsilon_{T,S,[K,\kappa]}^G \) form a complete set of orthogonal idempotents. Hence \(U(V(w)) = w\) for all \(w\). To establish the reverse identity \(V \circ U = \mathrm{id}\), consider a tuple \((w_{T,S, (K,\kappa)})_{(T,S, (K,\kappa))}\) in the direct sum. Then
\begin{eqnarray*}
&&VU\left( (w_{T,S, (K,\kappa)})_{(T,S, (K,\kappa))} \right)\\
&=&\left( \frac{1}{|N_G(T,S):T|} v_{T,S,[K,\kappa]}^G \sum_{(T',S'),[K',\kappa']} u_{T',S',[K',\kappa']}^G w_{T',S',[K',\kappa']} \right)_{(T,S,[K,\kappa])}.
\end{eqnarray*}
Now, since, by Theorem \ref{uvorthogonal}, we have
\[
v_{T,S,[K,\kappa]}^G \cdot u_{T',S',[K',\kappa']}^G = \delta_{(T,S,[K,\kappa]),(T',S',[K',\kappa'])} \cdot \Big(\sum_{g \in N_G(T,S)/T} \Iso(c_g)\Big)\varphi_{[K, \kappa]}^{T/S},
\]
and also $\varphi_{[K, \kappa]}^{T/S}\cdot w_{T,S,[K,\kappa]} = w_{T,S,[K,\kappa]}$ by the choice of the elements $w_{T,S,[K,\kappa]}$, we find that all cross-terms vanish, and what remains is
\[
VU\left( w_{T,S,[K,\kappa]} \right) = \left( \frac{1}{|N_G(T,S):T|} \sum_{g \in N_G(T,S)/T} \Iso(c_g)(w_{T,S,[K,\kappa]}) \right)_{(T,S,[K,\kappa])}.
\]
Because each \(w_{T,S,[K,\kappa]}\) is \(N_G(T,S)/T\)-invariant by construction, this averaging yields the original element. Therefore, \(VU = \mathrm{id}\), completing the verification.

As a consequence, each element \(w \in F(G)\) admits the decomposition
\[
w = \sum_{(T,S),[K,\kappa]} \epsilon_{T,S,[K,\kappa]}^G \cdot w.
\]
\end{nothing}

\begin{theorem}
    Let $k$ be a commutative ring, $F$ be a fibered biset functor over $k$ and $G$ be a finite group with $|G|\in k^\times$. There exists a decomposition 
\[
F(G) \cong \bigoplus_{(T,S) \in [\Pi(G)]} \bigoplus_{(K,\kappa) \in [\mathcal{M}_{T/S}^*]} \left( \varphi_{[K,\kappa]}^{T/S} F(T/S) \right)^{N_G(T,S)/T}.
\]
Here, $[\Pi(G)]$ denotes a set of representatives of the conjugacy classes of minimal sections of $G$, i.e., sections \(S\unlhd T\leq G\) with \(S \leq \Phi(T)\), and $[\mathcal{M}_{T/S}^*]$ denotes a set of orbit representatives $[K,\kappa]$ under the action of $\Out(T/S)$, where $K$ is a normal subgroup of $T/S$ contained in its Frattini subgroup, and $\kappa$ is a faithful character of $K$.

\end{theorem}

\section{Decomposing $\calF_{p,k}^{\ctimes}$, when $p\in k^{\times}$}
\label{Decomposing}
Let $k$ be a commutative unitary ring, and let $p \in k^{\times}$ be a fixed prime. In this section, we restrict to attention to the category of $\CC^{\times}$-fibered bisets on finite $p$-groups. The goal of this section is to decompose the category of $\ctimes$-fibered $p$-biset functors via idempotents defined earlier. This is a fibered version of Bouc's results Theorem 7.4 and Corollary 7.5 in \cite{double}. We begin by defining atoric $p$-groups.

\begin{defi}[Bouc]
   A finite \(p\)-group \(P\) is called \emph{atoric} if the following equivalent conditions hold.
   \begin{enumerate}
       \item \(P\) does not admit any decomposition \(P = E\times Q\), where \(E\) is a non-trivial elementary abelian group. \cite[Section 6.1]{double}
       \item For any non-trivial normal subgroup $N$ of \(P\), one has \(N\cap\Phi(P)\not = 1\). \cite[Lemma 6.3 (2)]{double}
       \item \(\Omega_1Z(P)\le \Phi(P)\). \cite[Lemma 6.3 (3)]{double}
   \end{enumerate}
\end{defi}

We denote by $\calA t_p$ the class of \textit{atoric} $p$-groups and by $[\calA t_p]$ a set of representatives of isomorphism classes in $\calA t_p$. 

Let $P$ be a finite $p$-group. As in \cite[Section 6]{double}, we denote the greatest atoric quotient of $P$ by $P^@$. By Proposition 6.4 of \cite{double}, there is a maximal normal subgroup $N$ of $P$ with the property that $N\cap \Phi(P)=1$ such that $P^@ = P/N$. We call $P^@$ the \emph{atoric part} of $P$. For further properties of atoric $p$-groups, we refer to Section~6 of \cite{double}. 

The following theorem refines a fundamental result of Bouc to the fibered setting, where the effect of the subgroup $K$ is made clear in an essential way. As in \cite{double}, this refinement plays a central role in establishing the decomposition theorem we prove below.

\begin{theorem}\label{thm:nonzeroatoric}
Let \( P \) and \( Q \) be finite \( p \)-groups, and let \( (T,S) \) and \( (V,U) \) be minimal sections of \( P \) and \( Q \), respectively. Suppose that 
\[
(K/S,\kappa) \in \calM_{T/S}^{*} \quad \text{and} \quad (L/U,\lambda) \in \calM_{V/U}^{*}.
\]
Then 
\[
\epsilon_{V,U,[L/U,\lambda]}^Q \cdot B_{k}^{\CC^{\times}}(Q,P) \cdot \epsilon_{T,S,[K/S,\kappa]}^P \neq \{0\} \, \text{implies that}\, (V/L)^@ \cong (T/K)^@. 
\]
\end{theorem}

\begin{proof}
Suppose that 
\[
\epsilon_{V,U,[L/U,\lambda]}^Q\cdot  B_{k}^{\CC^{\times}}(Q,P)\cdot  \epsilon_{T,S,[K/S,\kappa]}^P \neq \{0\},
\]
then there exists \( a\in B_{k}^{\mathbb{C}}(Q,P) \) such that \(\epsilon_{V,U,[L/U,\lambda]}^Q\cdot  a\cdot \epsilon_{T,S,[K/S,\kappa]}^P \not = 0\) and thus
\[
    \Indinf^Q_{V/U}\varphi_{[L/U,\lambda]}^{V/U}\Defres^Q_{V/U} \cdot a \cdot \Indinf^P_{T/S}\varphi_{[K/S,\kappa]}^{T/S}\Defres^P_{T/S} \neq 0.
\]
Thus, there exist \( (L'/U,\lambda')\in [L/U,\lambda] \) and \( (K'/S,\kappa')\in [K/S,\kappa] \) such that
\[
    \Indinf_{V/U}^Q \varphi_{L'/U,\lambda'}^{V/U}\Defres_{V/U}^Q \cdot a \cdot \Indinf_{T/S}^P\varphi_{K'/S,\kappa'}^{T/S}\Defres_{T/S}^P\neq 0. 
\]
In particular, the element 
\[
b=\Defres_{V/U}^Q \cdot a \cdot \Indinf_{T/S}^P
\] 
of \( B_{k}^{\CC^{\times}}(V/U,T/S) \) satisfies 
\[
\varphi_{L'/U,\lambda'}^{V/U} \cdot b \cdot \varphi_{K'/S,\kappa'}^{T/S} \neq 0.
\] 
This implies that there exists an element \( (N,\eta)\in \calM_{V/U\times T/S} \) such that
\begin{equation*}
    \varphi_{L'/U,\lambda'}^{V/U} \cdot \Big[\frac{(V/U)\times (T/S)}{N,\eta}\Big] \cdot \varphi_{K'/S,\kappa'}^{T/S} \neq 0.
\end{equation*}
Then, by Theorem~\ref{phiortogonal2}, the group $N$ must be covering in $(V/U) \times (T/S)$, and we must have
\[
    k_1(N) \cap \Phi(V/U) \leq L'/U \quad \text{and} \quad k_2(N) \cap \Phi(T/S) \leq K'/S.
\]
Hence, we obtain
\begin{eqnarray*}
    \frac{k_1(N)(L'/U)}{L'/U} \cap \Phi\left(\frac{V/U}{L'/U}\right)
    &=& \frac{k_1(N)(L'/U)}{L'/U} \cap \frac{\Phi(V/U)}{L'/U}=\\
    &=& \frac{k_1(N)(L'/U) \cap \Phi(V/U)}{L'/U}
    = \frac{L'/U}{L'/U},
\end{eqnarray*}
since $L'/U \in \Phi(V/U)$. Similarly, we also have
\[
    \frac{k_2(N)(K'/S)}{K'/S} \cap \Phi\left(\frac{T/S}{K'/S}\right) = \1.
\]
Therefore, by \cite[Proposition~6.6]{double}, it follows that
\[
    (V/L)^@ \cong \left(\frac{V/U}{L/U}\right)^@ \cong \left(\frac{V/U}{L'/U}\right)^@
    \cong \left(\frac{\frac{V/U}{L'/U}}{\frac{k_1(N)(L'/U)}{L'/U}}\right)^@
    \cong \left(\frac{V/U}{k_1(N)(L'/U)}\right)^@.
\]
By analogous reasoning on the right-hand side, we also have
\[
    (T/K)^@ \cong \left(\frac{T/S}{k_2(N)(K'/S)}\right)^@.
\]
For the above product to be nonzero, we must also have
\[
    e_{(L'/U,\lambda')}^{V/U} \cdot \left[\frac{(V/U)\times (T/S)}{N,\eta}\right] \cdot e_{(K'/S,\kappa')}^{T/S} \neq 0.
\]
This follows from the definitions of $\varphi_{L'/U,\lambda'}^{V/U}$, $\varphi_{K'/S,\kappa'}^{T/S}$, and Proposition \ref{eKtimeseL}. Moreover, by \cite[Proposition~4.2]{fibered biset}, we can compute the left and right invariants of this nonzero element. In particular,
\[
    e_{(L'/U,\lambda')}^{V/U} \cdot \left[\frac{(V/U)\times (T/S)}{N,\eta}\right] \cdot e_{(K'/S,\kappa')}^{T/S}
    = \left[\frac{(V/U)\times (T/S)}{N((L'/U)\times (K'/S)),\, \lambda' \cdot \eta \cdot \kappa'}\right],
\]
where
\begin{align*}
    p_1\big(N((L'/U)\times (K'/S))\big) &= V/U, \\
    p_2\big(N((L'/U)\times (K'/S))\big) &= T/S, \\
    k_1\big(N((L'/U)\times (K'/S))\big) &= k_1(N)(L'/U), \\
    k_2\big(N((L'/U)\times (K'/S))\big) &= k_2(N)(K'/S).
\end{align*}
Hence, we obtain the isomorphism
\[
    \frac{V/U}{k_1(N)(L'/U)} \cong \frac{T/S}{k_2(N)(K'/S)}.
\]
Combining this with the previous observations yields the desired isomorphism:
\[
    (V/L)^@ \cong (T/K)^@.
\]
\end{proof}

\begin{nothing}
 One may compare the above theorem with Lemma 6.3 of \cite{fibered biset} where it is proved that if a product 
 \[
 f_{\{K, \kappa\}}\cdot a\cdot f_{\{L,\lambda\}}
 \]
 for $a\in B_k^{\ctimes}(P,P)$ is non-zero then $(G, K,\kappa)$ is linked to $(G, L, \lambda)$. This also implies that $G/K$ is isomorphic to $G/L$. The above result shows that we loosen the condition considerably. 

 The following definition is our version of idempotents $b_L^P$ of \cite{double}. We change the notation to distinguish them properly. 
\end{nothing}

\begin{defi}\label{sec:cmPs}
    Let $M$ be an atoric $p$-group and $P$ be a finite $p$-group. The \emph{Bouc's idempotent associated to $M$ at $P$} is the element  $c_M^P$ of $B_k^{\CC^{\times}}(P,P)$ defined by 
    \[
        c_M^P = \sum_{\substack{(T,S)\in [\Pi(P)]\\ (K/S,\kappa)\in [\calM_{T/S}^{*}]\\ (T/K)^{@} \cong M}} \epsilon_{T,S,[K/S,\kappa]}^P.
    \]
\end{defi}
\begin{nothing}{\bf Properties of Bouc's idempotents.}
    It turns out that Bouc's idempotents $c_M^P$ replaces $b_L^P$ from \cite{double} in the best way, leading analogous results for the category of fibered $p$-biset functors. With all the preliminaries at hand, the proofs of properties of Bouc's idempotents are almost identical to those in \cite[Theorem 7.4]{double}. We shall omit certain details and list the properties as follows. 
    
    Let $M, N$ be atoric $p$-groups and $P, Q$ be finite $p$-groups.
\begin{enumerate}
    \item[(a)] The idempotent $c_M^P$ is non-zero if and only if $M$ is isomorphic to a subquotient of $P^@$. 
\end{enumerate}

The forward implication follows from the definition. For the converse, if $M\cong (T/S)^@$ for a minimal section $(T, S)$ of $P$,  then $\epsilon_{T, S, [S/S, 1]}^P$ is a summand of $c_M^P$.

\begin{enumerate}
    \item[(b)] If $c_M^Q\cdot B_k^{\CC^{\times}}(Q,P)\cdot c_N^P\neq {0}$, then $M\cong N$. 
\end{enumerate}

This is basically the previous theorem. Indeed, if the given subspace is non-zero then there are idempotents $
\epsilon_{V,U,[L/U,\lambda]}^Q$ and $\epsilon_{T,S,[K/S,\kappa]}^P$ such that $(V/L)^@\cong M$ and $(T/K)@\cong N$ and 
\[
\epsilon_{V,U,[L/U,\lambda]}^Q \cdot B_{k}^{\CC^{\times}}(Q,P) \cdot \epsilon_{T,S,[K/S,\kappa]}^P \neq \{0\}.
\]
Hence by Theorem \ref{thm:nonzeroatoric}, we get $M\cong (V/L)^@\cong (T/K)@\cong N$.

\begin{enumerate}
    \item[(c)] The sum of Bouc's idempotents $c_M^P$, as $M$ runs over all groups in $[\calA t_p]$, is equal to the identity element of $B_k^{\CC^{\times}}(P,P)$.
\end{enumerate}

As in the case of bisets, the idempotents $\epsilon_{T, S, [K, \kappa]}^P$ sum up to the identity, as $(T, S)$ runs over all minimal sections up to conjugation; $(K, \kappa)$ over all pairs in $\calM_{\Phi(T/S)}^{T/S}$, up to automorphisms of $T/S$. Now we can rearrange the terms by fixing the atoric part of $T/K$'s, we get the result.

\begin{enumerate}
    \item[(d)] For all $a\in B_k^{\CC^{\times}}(Q,P)$ we have
        \[
            c_M^Q\cdot a=a\cdot c_M^P.
        \]        
         
        \item[(e)] The family of elements $c_M^P\in B_k^{\CC^{\times}}(P,P)$, for finite $p$-groups $P$, is an idempotent endomorphism $c_M$ of the identity functor of the category $k\calC_p$. That is, it is an idempotent of the center of $k\calC_p$. The idempotent elements $c_M$, for $M\in [\calA t_p]$, are orthogonal, and their sum is equal to the identity element of the center of $k\calC_p$.
        \item[(f)]   For a given finite $p$-group $P$, the elements $c_M^P$, for $M\in [\calA t_p]$ such that $M\sqsubseteq P^@$, are non-zero orthogonal central idempotents of $B_k^{\CC^{\times}}(P,P)$, and their sum is equal to the identity of $B_k^{\CC^{\times}}(P,P)$.
    \end{enumerate}

Proofs of these properties are exactly the same as the proofs in \cite{double} of parts (3), (4) and (5) of Theorem 7.4. We omit the details.
\end{nothing}

\begin{nothing}
These properties lead to the decomposition of the category $\mathcal F_{p, k}^{\mathbb C^\times}$. Indeed, let $F$ be a $\ctimes$-fibered $p$-biset functor and $M$ be an atoric $p$-group.
As $P$ runs over all finite $p$-groups, the family $c_M^P\in B_k^{\CC^{\times}}(P,P)$ of idempotents defines a natural transformation 
\[
c_{M, F}: F\to F \, \text{given by}\, c_{M, F}(P)(w) = c_M^P\cdot w \, \text{for each}\, w\in F(P).
\]
Moreover if $F'$ is also a $\ctimes$-fibered biset functor, then for any natural transformation $\theta: F\to F'$ 
we have $$c_{M, F'}\circ \theta = \theta\circ c_{M, F}.$$ In particular $c_{M, ?}$ is an idempotent in the center of the category $\mathcal F_{p, k}^{\mathbb C^\times}$. Following Bouc, we denote this idempotent by $\widehat{c}_M$ and call it the \emph{Bouc's idempotent at $M$}. By the above properties of Bouc's idempotents, it is straightforward to show that as $M$ runs over all atoric $p$-groups up to isomorphism, we get a decomposition of the identity functor of the category $\mathcal F_{p, k}^{\mathbb C^\times}$ into a sum of orthogonal idempotents
\[
\id_{\mathcal F_{p, k}^{\mathbb C^\times}} = \sum_{M\in [\calA t_p]}\widehat{c}_M.
\]
Now writing $\widehat{c}_MF$ for the image of $c_{M, F}$, we obtain a canonical decomposition 
\[
F \cong \bigoplus_{M\in[\calA t_p]} \widehat{c}_MF
\]
for any $\ctimes$-fibered $p$-biset functor $F$. This decomposition is natural in $F$ and $\widehat{c}_MF$ inherits its functor structure from $F$.

This is exactly the fibered $p$-biset functor version of Bouc's decomposition of $p$-biset functors. We state the final result as a theorem for future reference.
\end{nothing}
\begin{theorem}[Block decomposition of $\mathcal{F}_{p,k}^{\mathbb{C}^\times}$]
\label{thm:block-decomposition}
Let $p$ be a prime number such that $p \in k^\times$. For each atoric $p$-group $M \in [\mathrm{At}_p]$, define the full subcategory
\[
\widehat{c}_M\mathcal{F}_{p,k}^{\mathbb{C}^\times} := \left\{ F \in \mathcal{F}_{p,k}^{\mathbb{C}^\times} \,\middle|\, \widehat{c}_M F = F \right\},
\]
consisting of those functors on which $\widehat{c}_M$ acts trivially. Then:

\begin{enumerate}[(1)]
    \item Each category $\widehat{c}_M\mathcal{F}_{p,k}^{\mathbb{C}^\times}$ is an abelian subcategory of $\mathcal{F}_{p,k}^{\mathbb{C}^\times}$, closed under taking subfunctors, quotients, and extensions.
    
    \item The family of orthogonal idempotents $\{ \widehat{c}_M \}_{M \in [\mathrm{At}_p]}$ induces an equivalence of categories
    \[
    \mathcal{F}_{p,k}^{\mathbb{C}^\times} \cong \prod_{M \in [\mathrm{At}_p]} \widehat{c}_M\mathcal{F}_{p,k}^{\mathbb{C}^\times},
    \]
    via the functor
    \[
    F \mapsto \left( \widehat{c}_M F \right)_{M \in [\mathrm{At}_p]}.
    \]
    
    \item Each functor $F \in \mathcal{F}_{p,k}^{\mathbb{C}^\times}$ admits a canonical decomposition
    \[
    F = \bigoplus_{M \in [\mathrm{At}_p]} \widehat{c}_M F.
    \]
\end{enumerate}
\end{theorem}

\begin{nothing}
    Having the above decomposition, a natural direction is to describe the blocks $\widehat{c}_M\calF_{p,k}^\ctimes$ for arbitrary atoric $p$-groups more explicitly. In \cite{double}, Bouc obtained detailed results on the structure of these blocks in a very elegant way. Although one might attempt to generalize his ideas to the case of fibered $p$-biset functors, we decided not to go in this direction since the results are expected to be much more complicated compared to the case of $p$-biset functors. While a complete classification of the fibered blocks is expected to be far more involved than in the unfibered case, we turn our attention to a coarser but structurally useful invariant: the vertex of an indecomposable functor.
    We shall apply this to obtain further results about fibered $p$-biset functors when we regard them as biset functors.
\end{nothing}

\begin{defi}
    Let $F$ be an indecomposable $\CC^{\times}$-fibered $p$-biset functor over $k$. The unique atoric $p$-group $M\in [\calA t_p]$ such that $F=\widehat{c}_MF$ is called the \textit{vertex} of $F$. 

    This terminology is in parallel to the theory of vertices for indecomposable $p$-biset functors in the sense of Bouc \cite[Definition 9.2]{double}.
\end{defi}
\begin{remarki}
\label{rem:vertex-ext}
Let $F$ and $F'$ be indecomposable objects in $\mathcal{F}_{p,k}^{\mathbb{C}^\times}$ with distinct vertices $M, N \in [\calA t_p]$. Then the extension groups between them vanish:
\[
\operatorname{Ext}^\ast_{\mathcal{F}_{p,k}^{\mathbb{C}^\times}}(F, F') = 0.
\]
This follows from the fact that $\mathcal{F}_{p,k}^{\mathbb{C}^\times}$ decomposes as a direct product of blocks $\widehat{c}_M \mathcal{F}_{p,k}^{\mathbb{C}^\times}$ for $M\in [\calA t_p]$.

\smallskip

Furthermore, we emphasize that the vertex of a functor does not necessarily coincide with a minimal group for the functor. The following theorem explores this relationship.
\end{remarki}

\begin{theorem}\label{thm:vertexindec}
Let $F$ be an indecomposable $\mathbb{C}^\times$-fibered $p$-biset functor over $k$ with vertex $M \in [\mathrm{At}_p]$, and suppose $P$ is a minimal group for $F$. Then there exists a normal subgroup $K \unlhd P$ such that
\[
K \leq \Phi(P) \cap Z(P), \quad \text{and} \quad M \cong (P/K)^@.
\]
\end{theorem}
\begin{proof}
Since $F = \widehat{c}_M F$, there must exist a summand $\epsilon^P_{T,S,[K/S,\kappa]}$ of $c_M^P$ such that
\[
\epsilon^P_{T,S,[K/S,\kappa]} \cdot F(P) \neq 0,
\]
for some minimal section $(T,S)$ of $P$ and a pair $(K/S, \kappa) \in [M^\ast_{T/S}]$ with $(T/K)^@ \cong M$.
By the definition of $\epsilon^P_{T,S,[K/S,\kappa]}$, we have
\[
\epsilon_{T, S, [K, \kappa]}^P = X\circ \Defres^P_{T/S} 
\] for a virtual fibered $(P, T/S)$-biset $X$. Since $P$ is minimal for $F$, this composition is nonzero only if $T/S \cong P$, hence $T = P$ and $S = 1$. Thus we obtain $$(P/K)^@\cong M.$$ Since $\kappa$ is a faithful character of $K$, it follows that $K$ is a central cyclic subgroup of $P$, and by its choice, it is contained in $\Phi(P)$. Hence the result.
\end{proof}
\begin{nothing}
    As an application of this theorem, we may determine vertices of simple $\ctimes$-fibered biset functors over $\CC$. By Theorem 9.2 of \cite{fibered biset}, these functors are parameterized by quadruples $(P, K, \kappa, [V])$ where $P$ is a finite $p$-group, $K\unlhd P$ is a cyclic normal subgroup of $P$ contained in $Z(P)\cap P'$ and $\kappa$ is a faithful character of $K$. The module $V$ is a simple module over the group algebra $\CC\Gamma_{(P, K,\kappa)}$. Here $\Gamma_{(P, K, \kappa)}$ is a finite group associated to the given triple, as in Section 6 of \cite{fibered biset}. The corresponding simple functor is denoted by $S_{(P, K, \kappa, [V])}$. 
\end{nothing}

\begin{theorem} \label{thm:vertexsimple}
The vertex of the simple $\CC^{\times}$-fibered $p$-biset functor $S_{(P,K,\kappa,V)}$ is $(P/K)^@$.
\end{theorem}
\begin{proof}
Since $S=S_{(P, K, \kappa, [V])}$ is indecomposable and $P$ is minimal for $S$, by Theorem \ref{thm:vertexindec}, the vertex of $S$ is $(P/N)^@$ for some normal subgroup $N$ of $P$ contained in $Z(P)\cap \Phi(P)$. We claim that $N$ can be chosen as $K$. 

By the construction of $S$ given in \cite[Section 9]{fibered biset}, we have that
\[
f_{(K,\kappa)}^P\cdot S(P)\not=0
\]
where $f_{(K,\kappa)}^P$ is as in Section \ref{sec:prelim}. See \cite[Section 4]{fibered biset} for details.
On the other hand, setting $M=(P/N)^@$, we must have
\[
c_M^PS(P) = S(P).
\]
As in the proof of Theorem \ref{thm:vertexindec}, this implies that there exists a pair $(N', \eta)$ such that $(P/N')^@ \cong M$ and
\[
\epsilon_{P, 1, [N', \eta]}^PS(P) \not= 0.
\]
Without loss of generality we put $N = N'$. By definition,
\[
\epsilon_{P, 1, [N, \eta]}^P = \varphi_{[N,\eta]}^P
\]
hence we must have 
\[
\varphi_{N, \eta}^P S(P)\not= 0.
\]
This implies that $\varphi_{N, \eta}^P\cdot f_{(K, \kappa)}^P\not=0$. Hence
\[
0\not= \varphi_{N, \eta}^P\cdot f_{(K, \kappa)}^P = \sum_{(N,\eta)\le (N', \eta')\in\calM_{\Phi(P)}^{P}} \mu_{\unlhd P}(N, N')\widetilde{e_P^P}\cdot e_{N', \eta'}^P\cdot f_{K, \kappa}^P
\]
By \cite[Proposition 5.6]{fibered biset}, we get that
$(N',\eta')\le(K,\kappa)$, which in particular implies that $N\le K$.

On the other hand, we have
\[
0\not= \varphi_{N, \eta}^P\cdot f_{(K, \kappa)}^P = \sum_{(K, \kappa)\le (K', \kappa')\in\calM_P^P}\mu_{\unlhd P}(K, K')\varphi_{N, \eta}^P\cdot e_{(K', \kappa')}^P
\]
In particular, $\varphi_{N, \eta}^P\cdot e_{(K', \kappa')}^P$ must be non-zero for some $(K', \kappa')$. By 
Theorem \ref{phiortogonal2}, we must have $K'\cap\Phi(P)\le N$. But $K\le K'\cap \Phi(P)$, hence we obtain $K\le N$. Together with the previous paragraph, we get $N=K$, as required.

\end{proof}

\begin{coro}
Let $S = S_{(P, K, \kappa, [V])}$ and $S' =S_{(Q, L, \lambda, [W])}$ be simple $\ctimes$-fibered $p$-biset functors.
 If $(P/K)^@\not \cong (Q/L)^@$, then $\Ext^*_{\calF_{p,k}}(S, S')=\{0\}$. 
\end{coro}

\begin{nothing}
    Next we consider composition factors of indecomposable functors with a given vertex. It is easy to see that the proof of \cite[Lemma 9.7]{double} holds when $F$ is chosen to be a $\ctimes$-fibered $p$-biset functor. Hence the following version of Theorem 9.8 of \cite{double} also holds. We omit the proof which is almost identical to the original proof of the cited theorem. 
\end{nothing}

\begin{theorem}
    Let $k$ be a field of characteristic different from $p$, and $M$ be an atoric $p$-group. Define
\[
\mathcal{F}_{p,k}^{\mathbb{C}^\times}[M] := \left\{ F \in \mathcal{F}_{p,k}^{\mathbb{C}^\times} \,\middle|\, \text{all composition factors of } F \text{ have vertex } M \right\}.
\]

    \begin{enumerate}
        \item For any $\ctimes$-fibered $p$-biset functor $F$, the subfunctor $\widehat{c}_MF$ is the greatest subfunctor of $F$ which belongs to $F_{p,k}^\ctimes[M]$.
        \item In particular the equality  $$\widehat{c}_M\calF_{p,k}^\ctimes=\calF_{p,k}^\ctimes[M]$$ holds.
    \end{enumerate}
\end{theorem}

\section{Applications}
\label{Applications}
In this section we consider two applications of the above theorem. The first is an adaptation of Bouc's example.

\begin{nothing} {\bf The monomial Burnside functor $B_k^\ctimes$.} By Theorem 11.2 of \cite{fibered biset}, the functor $B_k^\ctimes$ is indecomposable. Since
\[
B_k^\ctimes(1) = k \neq 0,
\]
its minimal group is the trivial group. By Theorem~\ref{thm:vertexindec}, we conclude that its vertex is also trivial.
It follows that all composition factors of $B_k^\ctimes$ must have vertex $1$. By Theorem~\ref{thm:vertexsimple}, any such simple functor must be of the form $S_{(P,K,\kappa,[V])}$ where:
\begin{enumerate}
    \item $P/K$ is elementary abelian,
    \item $K = \Phi(P) = P'$ is a central cyclic subgroup of $P$.
\end{enumerate}
In the case where \( P \) is non-abelian, we have that the derived subgroup \( P' \) is non-trivial. Since \( P' \leq Z(P) \), it follows that for all \( x, y \in P \), the identity
\[
[x, y]^p = [x^p, y]
\]
holds. But \( x^p \in \Phi(P) \leq Z(P) \), and hence \( [x^p, y] = 1 \). Therefore, \( [x, y]^p = 1 \), implying that each element of \( P' \) has order dividing \( p \). Since $P'$ is non-trivial and cyclic, it must have order \( p \). Groups satisfying these properties are referred to as \emph{generalized extraspecial \( p \)-groups} in the terminology by Stancu and have been classified in~\cite[Lemma 3.2]{S}. In particular, if \( P \) is a non-abelian \( p \)-group with central Frattini subgroup of order \( p \), then \( P \) is isomorphic to one of the following types:
\begin{table}[H]
\centering
\caption{Generalized extraspecial $p$-groups}
\begin{tabular}{|c|c|p{6.5cm}|}
\hline
\textbf{Notation} & \textbf{Center} & \textbf{Description} \\
\hline
\( X_{p^{2\ell+1}} \times E \) & \( \Phi(P) \times E \) & Direct product of an extraspecial \( p \)-group and an elementary abelian \( p \)-group. \\
\hline
\( (X_{p^{2\ell+1}} * C_{p^2}) \times E \) & \( C_{p^2} \times E \) & Direct product of an elementary abelian \( p \)-group with the central product of an extraspecial \( p \)-group and a cyclic group of order \( p^2 \). \\
\hline
\end{tabular}
\end{table}

Hence, any composition factor of \( B_k^\ctimes \) must be of the form \( S_{(P, \Phi(P), \kappa, [V])} \), where \( P \) belongs to one of the following classes:
\begin{itemize}
    \item Elementary abelian \( p \)-groups (in the abelian case),
    \item Generalized extraspecial \( p \)-groups with central Frattini subgroup of order \( p \), as described above.
\end{itemize}

We remark that in \cite[Theorem 15]{CY}, the set of composition factors of $B_k^{\ctimes}$ is determined. The minimal groups appear in these composition factors are all elementary abelian. 
\end{nothing}

\begin{nothing}{\bf Restriction from $\calF_{p,k}^\ctimes$ to $\calF_{p,k}$.} 
Let $F \in \mathcal{F}_{p,k}^{\mathbb{C}^\times}$ be a $\mathbb{C}^\times$-fibered biset functor. We can view $F$ as an ordinary biset functor by forgetting the fiber structure on morphisms:
\[
\mathcal{F}_{p,k}^{\mathbb{C}^\times} \longrightarrow \mathcal{F}_{p,k}, \quad F \mapsto F^\flat,
\]
where $F^\flat(P) = F(P)$ and the action of a biset $X \in B(Q,P)$ on $F(P)$ is defined via the obvious inclusion $kB(Q,P)\to B_k^{\mathbb{C}^\times}(Q,P)$.

Let $M \in [\mathrm{At}_p]$ be an atoric $p$-group, and suppose that $F \in \mathcal{F}_{p,k}^{\mathbb{C}^\times}$ satisfies $\widehat{c}_M F = F$, i.e., $F$ lies in the block corresponding to $M$. Then the underlying biset functor $F^\flat$ may decompose into summands with vertices different from $M$. To determine possible blocks which may contain summands of $F^\flat$, one needs to determine which idempotents $b_L^P$ from the classical biset category satisfy
\[
b_L^P \cdot F^\flat(P) \neq 0.
\]

In general, consider the block $\widehat{c}_M\calF_{p,k}^\ctimes$ and its restriction $(\widehat{c}_M\calF_{p,k}^\ctimes)^\flat$ as a subcategory of $\calF_{p,k}$ of $p$-biset functors. Then we obtain a decomposition 
\[
(\widehat{c}_M\calF_{p,k}^\ctimes)^\flat \cong \prod_{L\in [\calA t_p]} \widehat{b}_L \widehat{c}_M\calF_{p,k}^\ctimes.
\]
It is clear that a term $\widehat{b}_L \widehat{c}_M\calF_{p,k}^\ctimes$ is non-zero if and only if the composition $\widehat{b}_L \widehat{c}_M$ is non-zero. Our next aim is to determine all atoric $p$-groups $L$ for which $\widehat{b}_L \widehat{c}_M \not=0$ holds for a given atoric $p$-group $M$.
\end{nothing}

\begin{theorem}
    Let $M$ and $L$ be atoric $p$-groups. The following are equivalent.
    \begin{enumerate}
        \item $\widehat{b}_L\cdot \widehat{c}_M \not=0$
        \item There exists a finite $p$-group $Q$ and $K, S\unlhd Q$ such that 
        \begin{itemize}
            \item $K, S\le\Phi(Q)$, $K\cap S=\1$ and $K$ is central and cyclic,
            \item  the isomorphisms $(Q/K)^@\cong M$ and $(Q/S)^@\cong L$ hold.
        \end{itemize}
    \end{enumerate}
\end{theorem}
\begin{proof}
    Suppose the product $\widehat{b}_L\cdot \widehat{c}_M$ is non-zero. This holds if and only if there is a finite $p$-group $P$ such that $b_L^P\cdot c_M^P \not=0$. We evaluate this product.
    \begin{eqnarray*}
       0\not= b_L^P\cdot c_M^P &=& \sum_{\substack{(V,U)\in[\Pi(P)],\\ (K/U, \kappa)\in\calM_{(V/U)}^{*}\\ (V/K)^@\cong M}} b_L^P\cdot \epsilon_{V, U, [K/U, \kappa]}^P \\
        &=& \sum_{\substack{(V, U)\in[\Pi(P)],\\ (K/U, \kappa)\in\calM_{(V/U)}^{*}\\ (V/K)^@\cong M}} \frac{1}{|N_P(V, U):V|} b_L^P\cdot \Indinf_{V/U}^P \varphi_{[K/U, \kappa]}^{V/U}\Defres^P_{V/U} \\
        &=& \sum_{\substack{(V, U)\in[\Pi(P)],\\ (K/U, \kappa)\in\calM_{(V/U)}^{*}\\ (V/K)^@\cong M,\\ L \sqsubseteq (V/U)^@}} \frac{1}{|N_P(V, U):V|} \Indinf_{V/U}^P b_L^{V/U}\varphi_{[K/U, \kappa]}^{V/U}\Defres^P_{V/U} \\
    \end{eqnarray*}
    Here the last equality holds by \cite[Theorem 7.4 (3)]{double}. Now we concentrate on a non-zero middle term by expanding $b_L^{V/U}$ as in \cite[Notation 7.3]{double}:
    \begin{eqnarray*}
       0\not = b_L^{V/U}\cdot\varphi_{[K/U, \kappa]}^{V/U} &=& \sum_{\substack{(T,S)\in [\Pi(V/U)]\\ (T/S)^@\cong L}} \epsilon_{T, S}^{V/U} \cdot\big(\widetilde{e_{V/U}^{V/U}}\cdot \varphi_{[K/U, \kappa]}^{V/U}\big)\\
        &=& \sum_{\substack{(T,S)\in [\Pi(V/U)]\\ (T/S)^@\cong L}} \big(\epsilon_{T, S}^{V/U} \cdot\widetilde{e_{V/U}^{V/U}}\big)\cdot \varphi_{[K/U, \kappa]}^{V/U}
    \end{eqnarray*}
Here the idempotents $\epsilon_{T, S}^{V/U}$ are as defined in \cite[Notation 4.6]{double}, see also Section \ref{defn:etskk}. Now the product $\epsilon_{T, S}^{V/U} \cdot\widetilde{e_{V/U}^{V/U}}$ is non-zero only if $T = V/U$ by \cite[Lemma 2.18]{double} (cf. Section 3.8). Also in this case, the product is equal to $\varphi_{S}^{V/U}$ (see \cite[Notation 3.4]{double} or Equation (\ref{Boucsphi})). Hence the sum becomes
 \begin{eqnarray*}
       0\not= b_L^{V/U}\cdot\varphi_{[K/U, \kappa]}^{V/U} 
        &=& \sum_{\substack{(V/U,S)\in [\Pi(V/U)]\\ ((V/U)/S)^@\cong L}} \varphi_S^{V/U}\cdot \varphi_{[K/U, \kappa]}^{V/U}
    \end{eqnarray*}
Once more we choose a non-zero term, say $\varphi_S^{Q}\cdot \varphi_{K/U, \kappa}^{Q}$, where $Q = V/U$. By \cite[Proposition 3.6 (1)]{double} and Proposition phikappaisoThis is equal to
Once more we choose a non-zero term, say $\varphi_S^{Q}\cdot \varphi_{K/U, \kappa}^{Q}$, where $Q = V/U$. By \cite[Proposition 3.6 (1)]{double} and part (1) of Proposition \ref{phikappaiso}, this is equal to
\begin{eqnarray}\label{eqn:phitimesphi}
     \sum_{\substack{(K/U,\kappa)\le (K',\kappa')\in\calM_{\Phi(Q)}^Q\\S'\unlhd Q,\, S\le S'\le\Phi(Q)}}\mu_{\unlhd Q}(K/U, K')\mu_{\unlhd}(S, S')\widetilde{e_Q^Q}\cdot \big( e_{(K',\kappa')}^Q\cdot e_{(S', {\bf 1})}^Q\big)\cdot \widetilde{e_Q^Q}.
    \end{eqnarray}
Finally since the above sum is non-zero, then
\[
e_{(K',\kappa')}^Q\cdot e_{(S', {\bf 1})}^Q \not=0
\]
for some indices $(K',\kappa')$ and $S'$. By \cite[Proposition 4.2]{fibered biset}, this holds if and only if $K'\cap S'\le \ker\kappa'$. Since $\kappa$ is faithful and $\kappa'|_K = \kappa$, the inequality implies $K\cap S = 1$. 
Combining all the conditions appear in the indices, we see that all the required conditions are satisfied for the triple $(Q, K, S)$.

Conversely, suppose that there exists a finite $p$-group $Q$ together with normal subgroups $K, S \unlhd Q$ such that $K, S \le \Phi(Q)$, $K \cap S = 1$, $K$ is central and cyclic, and the isomorphisms $(Q/K)^@ \cong M$ and $(Q/S)^@ \cong L$ hold. Choose a faithful character $\kappa \in K^\ast$. Then the fibered idempotent $\varphi^Q_{[K,\kappa]}$ appears in the decomposition of $\widehat{c}_M$, and the idempotent $\varphi^Q_S$ appears in the decomposition of $\widehat{b}_L$. Since $K \cap S = 1$, by \cite[Proposition 4.2]{fibered biset}, $e_{(K, \kappa)}^Q\cdot e_{(S, \bf 1)}^Q = e_{(KS, \kappa\bf 1)}^Q$ is non-zero and this term survives in the sum (\ref{eqn:phitimesphi}). Hence 
\[
\varphi^Q_S \cdot \varphi^Q_{[K,\kappa]} \ne 0.
\]
which implies that the product $\widehat{b}_L \cdot \widehat{c}_M$ is also nonzero, as it contains a nonzero contribution from the group $Q$.

\end{proof}

\begin{defi}\label{defn:centralresolution}
    Let $M$ and $L$ be atoric $p$-groups. We say that $L$ is a \emph{central resolution} of $M$, written $M \rightsquigarrow L$, if there is a finite $p$-group $Q$ and normal subgroups $K, S\unlhd Q$ such that
    \begin{enumerate}
        \item $K, S\le \Phi(Q)$,
        \item $K\cap S= {\bf 1}$,
        \item $K$ is central and cyclic,
        \item $(Q/K)^@\cong M$,
        \item $(Q/S)^@\cong L$.
    \end{enumerate}
\end{defi}
This definition keeps track of atoric $p$-groups related to a given atoric $p$-group $M$ via the above theorem. One does not expect this relation to be symmetric since we are aiming to resolve the structure of a block of fibered biset functors. However it is possible to refine the definition by replacing $Q$ with an atoric $p$-group. 
\begin{lem}
    Let $M$ and $L$ be atoric $p$-groups. Suppose $L$ is a central resolution of $M$ via the triple $(Q, K, S)$. Then there is a triple $(T, K', S')$ where $T$ is an atoric $p$-group such that $M \rightsquigarrow L$ via $(T, K', S')$.
\end{lem}
\begin{proof}
    By \cite[Proposition 6.4]{double}, we write $Q\cong N\times T$ where $T = Q^@$ is atoric and $N$ is a normal, elementary abelian subgroup of $Q$ such that $N\cap \Phi(Q) = {\bf 1}$. Let $q: Q\to T$ be the surjective homomorphism with kernel $N$ and set $K' = q(K)$ and $S' = q(S)$. We claim that the triple $(T, K', S')$ works as required in the lemma. We check the conditions. 
    \begin{enumerate}
        \item Since $K$ and $S$ are contained in $\Phi(Q)$, we have that $K' , S'\le q(\Phi(Q))$. But $q(\Phi(Q))\le \Phi(N\times T)$. Hence the first condition holds.
        \item The second and third conditions are clearly satisfied.
        \item To prove that $(T/K')^@\cong M$, consider the homomorphism 
        \[
        \pi: Q/K \to T/K', xK\mapsto q(x)K'
        \]
        This is a well-defined homomorphism and its kernel is equal to $NK/K$. We have 
        \[
        NK/K\cap \Phi(Q/K) = (NK\cap \Phi(Q))/K = K/K.
        \]
        Here the first equality holds since $K\le \Phi(Q)$ and the second holds since $N\cap K = {\bf 1}$. Thus by \cite[Proposition 6.6]{double}, we get $(T/K')^@ \cong (Q/K)^@ \cong M$. Similarly, one can prove that $(T/S')^@\cong (Q/S)^@ \cong L$.
    \end{enumerate}
\end{proof}
\begin{coro}
    Let $M$ be an atoric $p$-group and $L$ be an atoric $p$-group which is a cyclic central Frattini extension of $M$, that is, there is a short exact sequence
    \[
\begin{tikzcd}
0 \arrow[r] & K \arrow[r] & L \arrow[r] & M \arrow[r] & 0
\end{tikzcd}
\] such that $K\le Z(L)\cap \Phi(L)$ is cyclic. Then $M \rightsquigarrow L$. 
\end{coro}
\begin{proof}
This follows because the triple $(L, K, {\bf 1})$ satisfies the conditions of Definition \ref{defn:centralresolution}.
\end{proof}
\begin{remarki}
    It turns out that not all triples $(T, K, S)$ can be reduced to the case where $S$ is trivial. Indeed, given an atoric $p$-group $M$ and a non-trivial normal subgroup $S$ of $M$ contained in $\Phi(M)$, then $M\rightsquigarrow M/S$ via the triple $(M, {\bf 1}, S)$. Since $|M/S| < |M|$, there is no triple $(T, K, {\bf 1})$ making $M/S$ a central resolution of $M$.
\end{remarki}

Finally we obtain the following necessary condition for composition factors of $F^\flat$ for a fibered $p$-biset functor $F$. As an application we shall consider the restriction of the simple functor $S_{(\1, \1, \1, [\mathbb C])}$.

\begin{coro}
    Let $F$ be an indecomposable fibered $p$-biset functors with vertex $T$. Then a simple biset functor $S_{Q, W}$ is a composition factor of $F^\flat$ only if $Q^@\rightsquigarrow T$.
\end{coro}
\begin{proof}
    Suppose $S_{Q,W}$ is a composition factor of $F^\flat$. Then there is an indecomposable summand $M$ of $F^\flat$ having $S_{Q, W}$ as a composition factor. Let $L$ be the vertex of $M$. Then by \cite[Theorem 9.8]{double}, we must have $b_L M=M$ and also that $Q^@ = L$. Now since $F$, as a fibered biset functor, has vertex $T$, we also have $c_TF = F$ which implies $(b_L\cdot c_T) F \not= 0$. Hence $b_L\cdot c_T\not= 0$. Now the result follows from the definition of being a central resolution.  
\end{proof}
\begin{nothing} {\bf Central resolution of {\bf 1}.} Next we determine atoric $p$-groups $L$ such that ${\bf 1}\rightsquigarrow L$. By definition ${\bf 1}\rightsquigarrow L$ holds if there is an atoric $p$-group $T$ and normal subgroups $K$ and $S$ satisfying the conditions in Definition \ref{defn:centralresolution}. Since $(T/K)^@= {\bf 1}$, we get that $T/K$ is elementary abelian, and $K = \Phi(T)$, as in the previous section. In particular, the Frattini subgroup $\Phi(T)$ is cyclic and central. Since $T$ is atoric, we also have that $\Omega_1(Z(T))\le \Phi(T)$. Since the latter is cyclic, the center $Z(T)$ must also be cyclic. Hence we get that $T$ is an atoric $p$-group with cyclic center whose Frattini is also central. On the other hand, since $S\le \Phi(T)$ and $\Phi(T)\cap S = {\bf 1}$, we get $S={\bf 1}$. Now we conclude the following.
\end{nothing}

\begin{coro}
    An atoric $p$-group $L$ is a central resolution of ${\bf 1}$ if and only if $L$ has a cyclic center and central Frattini subgroup.
\end{coro}

The groups described in the above corollary are explicitly classified by Bornand~\cite{bornand}. If \(L\) is abelian, then it is necessarily cyclic. Bornand refers to non-abelian \(p\)-groups \(P\) with cyclic center \(Z(P)\) and central Frattini subgroup \(\Phi(P)\) as \emph{quasi-extraspecial}. In~\cite[Proposition~1.3.22]{bornand}, a classification of all such groups is given in the case where \(|\Phi(P)| > p\).

In the case where \(\Phi(P) \cong C_p\) and \(Z(P) \geq \Phi(P)\) we use~\cite[Lemma~3.2]{S}, together with the fact that \(Z(P)\) is cyclic, it follows that \(P\) is either isomorphic to an extraspecial \(p\)-group or to a central product of an extraspecial \(p\)-group with a cyclic group of order \(p^2\).

Hence, if \(L\) is a non-abelian \(p\)-group with cyclic center \(Z(L)\) and central Frattini subgroup \(\Phi(L)\), then \(L\) is isomorphic to one of the following:
\begin{table}[H]
\caption{Quasi-extraspecial $p$-groups}
\renewcommand{\arraystretch}{1.2}
\begin{tabular}{|c|c|c|p{6cm}|}
\hline
\textbf{Notation} & \textbf{Frattini} & \textbf{Center} & \textbf{Description} \\
\hline
\( X_{p^{2\ell+1}} \) & \( C_p \) & \( C_p \) & An extraspecial \( p \)-group. \\
\hline
\( X_{p^{2\ell+1}} * C_{p^2} \) & \( C_p \) & \( C_{p^2} \) & Central product of an extraspecial \( p \)-group with a cyclic group of order \( p^2 \). \\
\hline
\( X^{+}_{p^{2\ell+1}} * C_{p^{m+1}} \) & \( C_{p^m} \) & \( C_{p^{m+1}} \) & Central product of an extraspecial \( p \)-group of exponent \( p \) with a cyclic group of order \( p^{m+1} \), for \( m > 1 \). \\
\hline
\( X^{+}_{p^{2\ell-1}} * M_{p^{m+2}} \) & \( C_{p^m} \) & \( C_{p^m} \) & Central product of an extraspecial \( p \)-group of exponent \( p \) with the group \( M_{p^{m+2}} \), for \( m > 1 \). \\
\hline
\end{tabular}

\medskip

We note that the group \( M_{p^{m+2}} \) has presentation 
\[M_{p^{m+2}} = \langle x, y \mid x^p = y^{p^{m+1}} = 1,\ [x,y] = y^{p^m} \rangle.\] In the case \( p = 2 \), the notation \( X^{+}_{p^{2\ell+1}} \) denotes the central product \( D_8^{*\ell} \), that is, the central product of \( \ell \) copies of the dihedral group of order \( 8 \).
\end{table}

\begin{example}
Consider the simple fibered $p$-biset functor $S_{\1, \1, \1, [\CC]}$. By \cite[Section 11B]{fibered biset}, it is isomorphic to the functor of complex characters $\CC R_\CC$. By the above corollary, the composition factors of the underlying biset functor $S_{\1, \1, \1, [\CC]}^\flat\cong \CC R_\CC$ must be of the form $S_{Q, V}$ where $Q$ is a central resolution of the trivial group. On the other hand, by \cite[Section 7]{biset functor}, we have the decomposition
\[
\CC R_\CC \cong \bigoplus_{C,\sigma} S_{C, \CC_{\sigma}}
\]
where the sum runs over all cyclic groups $C$ and all primitive $\CC[\Aut(C)]$-characters $\sigma$. This example illustrates that the actual composition factors may involve only a small subclass of the groups appearing as central resolutions. However, as the next example shows, not all central resolutions are redundant.
\end{example}

\begin{example}
We consider the simple fibered biset functor $S$ mentioned in \cite[Remark 9.8]{fibered biset}. Consider the quadruple $(D_8, Z, \zeta, [V])$ where $D_8$ is the dihedral group of order $8$, $Z$ denotes its center, $\zeta$ is the non-trivial homomorphism $Z\to \CC^\times$ and $V$ is a simple $\CC\Gamma_{(D_8, Z,\zeta)}$-module. For this example we do not need the explicit structure of $V$. 

Let $S = S_{(D_8, Z, \zeta, [V])}$. The quadruple $(Q_8, Z',\zeta', [V'])$, where $Q_8$ is the group of quaternion group of order $8$, $Z'$ is its center, $\zeta'$ is the non-trivial homomorphism $Z'\to \ctimes$, and $V'$ corresponding simple module, is linked to $(D_8, Z, \zeta, [V])$ and hence
\[
S = S_{(D_8, Z, \zeta, [V])}\cong S_{(Q_8, Z', \zeta', [V'])}
\]
by \cite[Remark 9.8]{fibered biset}. Now by Theorem \ref{thm:vertexsimple}, the vertex of $S$ is $\1$ since $(D_8/Z)^@ =\1$. It follows that any composition factor of $S^\flat$ must be of the form $S_{Q, V}$ where $Q$ is a central resolution of $\1$.

On the other hand, both $D_8$ and $Q_8$ are minimal groups for the biset functor $S^\flat$. Hence it has composition factors of the form $S_{D_8, U}$ and $S_{Q_8, U'}$ for some simple $\CC[\Out(D_8)]$-module $U$ and some simple $\CC[\Out(Q_8)]$-module $U'$. Note that $D_8$ and $Q_8$ are both atoric $2$-groups and hence are central resolutions of $\mathbf{1}$.  
\end{example}

\end{document}